\documentclass[11pt,english]{article}

\usepackage[T1]{fontenc}
\usepackage[latin9]{inputenc}
\usepackage{geometry}
\usepackage{babel}
\usepackage{amsmath}
\usepackage{amsthm}
\usepackage{amssymb}
\usepackage{color}
\usepackage[unicode=true,pdfusetitle,
 bookmarks=true,bookmarksnumbered=false,bookmarksopen=false,
 breaklinks=false,pdfborder={0 0 0},pdfborderstyle={},backref=false,colorlinks=false]
 {hyperref}

\makeatletter

\usepackage[nameinlink,capitalise,noabbrev]{cleveref}

\hypersetup{%
    bookmarksnumbered, bookmarksopen=true, bookmarksopenlevel=1,%
}

\theoremstyle{plain}
\newtheorem{thm}{Theorem}[section]
\crefname{thm}{Theorem}{Theorems}
\theoremstyle{plain}
\newtheorem{lem}[thm]{Lemma}
\crefname{lem}{Lemma}{Lemmas}
\theoremstyle{plain}

\theoremstyle{plain}
\newtheorem*{claim*}{Claim}
\crefname{claim}{Claim}{Claims}
\theoremstyle{definition}

\theoremstyle{plain}

\theoremstyle{plain}
\newtheorem{prop}[thm]{Proposition}
\theoremstyle{definition}

\theoremstyle{definition}

\theoremstyle{plain}
\newtheorem{claim}[thm]{Claim}
\crefformat{equation}{#2(#1)#3}

\date{}

\usepackage{scalerel}

\usepackage{bm}

\usepackage{bbm}

\usepackage[nameinlink,capitalise,noabbrev]{cleveref}
\crefname{appsec}{Appendix}{Appendices}
\crefname{enumi}{condition}{conditions}
\Crefname{enumi}{Condition}{Conditions}

\let\originalleft\left
\let\originalright\right
\renewcommand{\left}{\mathopen{}\mathclose\bgroup\originalleft}
\renewcommand{\right}{\aftergroup\egroup\originalright}
\usepackage{pgfplots}
\usetikzlibrary{pgfplots.groupplots}
\usepackage{verbatim}

\makeatletter
\renewcommand*{\UrlTildeSpecial}{%
  \do\~{%
    \mbox{%
      \fontfamily{ptm}\selectfont
      \textasciitilde
    }%
  }%
}%
\let\Url@force@Tilde\UrlTildeSpecial
\makeatother

\makeatother

\begin{document}

\title{\texorpdfstring{\vspace{-1cm}}{}
Proof of a conjecture
on induced
subgraphs of Ramsey graphs}

\author{Matthew Kwan \thanks{Department of Mathematics, Stanford University, Stanford, CA 94305. Email: \href{mailto:mattkwan@stanford.edu} {\nolinkurl{mattkwan@stanford.edu}}. This research was done while the author was working at ETH Zurich, and is supported in part by SNSF project 178493.}\and
Benny Sudakov\thanks{Department of Mathematics, ETH, 8092 Z\"urich, Switzerland. Email:
\href{mailto:benjamin.sudakov@math.ethz.ch} {\nolinkurl{benjamin.sudakov@math.ethz.ch}}.
Research supported in part by SNSF grant 200021-175573.}}

\maketitle
\global\long\def\RR{\mathbb{R}}
\global\long\def\QQ{\mathbb{Q}}
\global\long\def\HH{\mathbb{H}}
\global\long\def\E{\mathbb{E}}
\global\long\def\Var{\operatorname{Var}}
\global\long\def\CC{\mathbb{C}}
\global\long\def\NN{\mathbb{N}}
\global\long\def\ZZ{\mathbb{Z}}
\global\long\def\GG{\mathbb{G}}
\global\long\def\BB{\mathbb{B}}
\global\long\def\DD{\mathbb{D}}
\global\long\def\cL{\mathcal{L}}
\global\long\def\supp{\operatorname{supp}}
\global\long\def\one{\mathbbm{1}}
\global\long\def\range#1{\left[#1\right]}
\global\long\def\d{\operatorname{d}\!}
\global\long\def\falling#1#2{\left(#1\right)_{#2}}
\global\long\def\f{\mathbf{f}}
\global\long\def\im{\operatorname{im}}
\global\long\def\sp{\operatorname{span}}
\global\long\def\rank{\operatorname{rank}}
\global\long\def\sign{\operatorname{sign}}
\global\long\def\mod{\operatorname{mod}}
\global\long\def\id{\operatorname{id}}
\global\long\def\disc{\operatorname{disc}}
\global\long\def\lindisc{\operatorname{lindisc}}
\global\long\def\tr{\operatorname{tr}}
\global\long\def\adj{\operatorname{adj}}
\global\long\def\Unif{\operatorname{Unif}}
\global\long\def\Po{\operatorname{Po}}
\global\long\def\Bin{\operatorname{Bin}}
\global\long\def\Ber{\operatorname{Ber}}
\global\long\def\Geom{\operatorname{Geom}}
\global\long\def\sat{\operatorname{sat}}
\global\long\def\Hom{\operatorname{Hom}}
\global\long\def\vol{\operatorname{vol}}
\global\long\def\floor#1{\left\lfloor #1\right\rfloor }
\global\long\def\ceil#1{\left\lceil #1\right\rceil }
\global\long\def\cond{\,\middle|\,}
\global\long\def\n#1{n_{\scalerel*{#1}{>}}}

\begin{abstract}
An $n$-vertex graph is called \emph{$C$-Ramsey} if it has no clique
or independent set of size $C\log n$. All known constructions of
Ramsey graphs involve randomness in an essential way, and there is
an ongoing line of research towards showing that in fact all Ramsey
graphs must obey certain ``richness'' properties characteristic
of random graphs. More than 25 years ago, Erd\H os, Faudree and S\'os conjectured that in any $C$-Ramsey graph there
are $\Omega\left(n^{5/2}\right)$ induced subgraphs, no pair of which
have the same numbers of vertices and edges. Improving on earlier results of Alon, Balogh, Kostochka and Samotij, in this paper we prove this conjecture.
\end{abstract}

\section{Introduction}

An induced subgraph of a graph is said to be \emph{homogeneous }if
it is a clique or independent set. A classical result in Ramsey theory,
proved in 1935 by Erd\H os and Szekeres~\cite{ES35}, is that every
$n$-vertex graph has a homogeneous subgraph with at least $\frac{1}{2}\log_{2}n$
vertices. On the other hand, Erd\H os~\cite{Erd47} famously used
the probabilistic method to prove that, for all $n$, there exists an $n$-vertex graph with no
homogeneous subgraph on $2\log_{2}n$ vertices. Despite significant
effort (see for example~\cite{FW81,BRSW12,Coh16,CZ16}), there are
no non-probabilistic constructions of graphs with comparably small
homogeneous sets. 

For some fixed $C$, say an $n$-vertex graph is \emph{$C$-Ramsey} if it has no
homogeneous subgraph of size $C\log_{2}n$. It is widely believed
that $C$-Ramsey graphs must in some sense resemble random graphs,
and this belief has been supported by a number of theorems showing
that certain ``richness'' properties characteristic of random graphs
hold for all $C$-Ramsey graphs. The first result of this type was
due to Erd\H os and Szemer\'edi~\cite{ES72}, who showed that $C$-Ramsey
graphs have density bounded away from 0 and 1. This basic result was
the foundation for a large amount of further research; over the years
many conjectures have been proposed and resolved as our understanding
of Ramsey graphs has improved. Improving a result of Erd\H os and
Hajnal~\cite{EH77}, Pr\"omel and R\"odl~\cite{PR99} proved that
for every constant $C$ there is $c>0$ such that every $n$-vertex
$C$-Ramsey graph contains every possible graph on $c\log_{2}n$ vertices
as an induced subgraph. Shelah~\cite{She98} proved that every $n$-vertex
$C$-Ramsey graph contains $2^{\Omega\left(n\right)}$ non-isomorphic
induced subgraphs. Answering a question of Erd\H os, Faudree and
S\'os~\cite{Erd92,Erd97}, Bukh and Sudakov~\cite{BS07} showed that
every $n$-vertex $C$-Ramsey graph has an induced subgraph with $\Omega\left(\sqrt{n}\right)$
different degrees. 

Despite this progress, there are several problems that have remained open
for quite some time. Two of them deal with the variation in the numbers of edges
and vertices of induced subgraphs of Ramsey graphs. For a graph $G$, let 
\begin{align*}
\Phi\left(G\right) & =\left\{ e\left(H\right):H\text{ is an induced subgraph of }G\right\} .
\end{align*}
Erd\H os and McKay~\cite{Erd92,Erd97} conjectured that for any $C$
there is $\delta>0$ such that for every $n$-vertex $C$-Ramsey graph
$G$, the set $\Phi\left(G\right)$ contains the interval $\left\{ 0,\dots,\delta n^{2}\right\} $.
Progress on this conjecture has come from two directions. First, Alon,
Krivelevich and Sudakov~\cite{AKS03} proved a weaker result with
$n^{\delta}$ in place of $\delta n^{2}$. Second, recently Narayanan,
Sahasrabudhe and Tomon~\cite{NST16} proposed a natural relaxation
of the Erd\H os--McKay conjecture, that $\Phi\left(G\right)$ contains
at least $\Omega\left(n^{2}\right)$ values (not necessarily forming an interval). 
They showed that $\left|\Phi\left(G\right)\right|=n^{2-o\left(1\right)}$, and
in~\cite{KS} we proved their conjecture that Ramsey graphs induce subgraphs of quadratically many sizes.

Next, for a graph $G$ let
\[
\Psi\left(G\right)=\left\{ \left(v\left(H\right),e\left(H\right)\right):H\text{ is an induced subgraph of }G\right\} .
\]
Strengthening a conjecture of Alon and Bollob\'as, it was conjectured
by Erd\H os, Faudree and S\'os that for any
fixed $C$ and any $n$-vertex $C$-Ramsey graph $G$, we have $\left|\Psi\left(G\right)\right|=\Omega\left(n^{5/2}\right)$. This problem appeared in several of Erd\H os' problem papers~\cite{Erd92,Erd92b,Erd97}.
Of course, since $\left|\Psi\left(G\right)\right|\ge\left|\Phi\left(G\right)\right|$,
our result in~\cite{KS} implies that $\left|\Psi\left(G\right)\right|=\Omega\left(n^{2}\right)$, which
was also proved much earlier by Alon and Kostochka~\cite{AK09}.
Until now, the best progress on the Erd\H os--Faudree--S\'os
conjecture was due to Alon, Balogh, Kostochka and Samotij~\cite{ABKS09},
who proved it with the exponent $2.369$ in place of $5/2$.

In this paper we establish the Erd\H os--Faudree--S\'os
conjecture, combining ideas from many of the aforementioned papers.
\begin{thm}
\label{conj:EFS}For any fixed $C>0$, there is $\gamma>0$ such that
every $n$-vertex $C$-Ramsey graph $G$ has $\left|\Psi\left(G\right)\right|=\gamma n^{5/2}$.
\end{thm}

As mentioned in~\cite{Erd92b}, we remark that the order of magnitude $n^{5/2}$ is best-possible.
This can be seen by considering a random graph $\GG\left(n,1/2\right)$
where each edge is present independently with probability $1/2$ (it
is well known that this is an $O\left(1\right)$-Ramsey graph with
probability $1-o\left(1\right)$). Briefly, one can use a concentration
inequality to show that with probability $1-o\left(2^{n}\right)$,
the number of edges in any fixed vertex subset of $\GG\left(n,1/2\right)$
lies in an interval of length $O\left(n^{3/2}\right)$, and by the
union bound it follows that with probability $1-o\left(1\right)$,
for each $0\le\ell\le n$ there are at most $O\left(n^{3/2}\right)$
different numbers of edges among $\ell$-vertex induced subgraphs.
This proves that $\left|\Psi\left(\GG\left(n,1/2\right)\right)\right|=O\left(n^{5/2}\right)$
with probability $1-o\left(1\right)$. See also~\cite[Section~4]{AK09}
for further discussion of $\Psi\left(\GG\left(n,1/2\right)\right)$.

The rest of the paper is organised as follows. In \cref{sec:discussion} we give a very high-level outline of the basic ideas of our proof and briefly compare it to previous work. In \cref{sec:basic-tools} we collect a number of basic tools which will be useful for our proof (some of which are standard, and some of which are new), and in  \cref{sec:EFS-proof} we present the technical details of our proof. Finally, in \cref{sec:concluding} we discuss some potential further directions of research.

\subsection{Notation and basic definitions}\label{subsec:definitions}

We use standard asymptotic notation throughout. For functions $f=f\left(n\right)$
and $g=g\left(n\right)$ we write $f=O\left(g\right)$ to mean that there
is a constant $C$ such that $\left|f\right|\le C\left|g\right|$,
we write $f=\Omega\left(g\right)$ to mean there is a constant $c>0$
such that $f\ge c\left|g\right|$ for sufficiently large $n$, we
write $f=\Theta\left(g\right)$ to mean that $f=O\left(g\right)$
and $f=\Omega\left(g\right)$, and we write $f=o\left(g\right)$ or
$g=\omega\left(f\right)$ to mean that $f/g\to0$ as $n\to\infty$. All asymptotics
are as $n\to\infty$ unless stated otherwise. Floor and ceiling symbols
will be systematically omitted where they are not crucial.

For two multisets $A$ and $B$, let $A\triangle B$ be the set of
elements which have different multiplicities in $A$ and $B$ (so
if $A$ and $B$ are ordinary sets, then $A\triangle B$ is the ordinary
symmetric difference $\left(A\setminus B\right)\cup\left(B\setminus A\right)$).
For a set $A$, we denote by $\binom{A}{k}$ the set of all $k$-subsets
of elements of $A$.

We also use standard graph theoretic notation throughout. In particular,
in a graph, $e\left(A\right)$ is the number of edges which are contained
inside a vertex subset $A$, and $e\left(A,B\right)$ is the number
of edges between two disjoint vertex subsets $A$ and $B$. For a
vertex $v$ and a set of vertices $A$, we denote the set of neighbours
of $v$ in $A$ by $N_{A}\left(v\right)=N\left(v\right)\cap A$ and
we denote the degree of $v$ into $A$ by $d_{A}\left(v\right)=\left|N_{A}\left(v\right)\right|$.

We also make some less standard graph theoretic definitions that will
be convenient for the proof. For a set of vertices $\boldsymbol{v}=\left\{ v_{1},\dots,v_{k}\right\} $,
let $N\left(\boldsymbol{v}\right)$ (respectively $N_{U}\left(\boldsymbol{v}\right)$)
be the multiset union of $N\left(v_{1}\right),\dots,N\left(v_{k}\right)$
(respectively, of $N_{U}\left(v_{1}\right),\dots,N_{U}\left(v_{k}\right)$.
Let $d\left(v\right)=d\left(v_{1}\right)+\dots+d\left(v_{k}\right)$
(respectively $d_{U}\left(\boldsymbol{v}\right)=d_{U}\left(v_{1}\right)+\dots+d_{U}\left(v_{k}\right)$)
be the size of $N\left(\boldsymbol{v}\right)$ (respectively, of $N_{U}\left(\boldsymbol{v}\right)$),
accounting for multiplicity.

Finally, we remark that we will often use variable names of the form
$\n A$ to denote the size of a set $A$. (This is really only a convention, not a definition; we will often introduce $\n A$ before the set $A$ has actually been defined).

\section{Discussion and main ideas of the proof}\label{sec:discussion}

According to Erd\H os~\cite{Erd92}, at the time the problem was
proposed, he and S\'os had already proved the weaker bound that $\left|\Psi\left(G\right)\right|=\Omega\left(n^{3/2}\right)$
for $O\left(1\right)$-Ramsey graphs. In fact, there are at least
two reasonably simple ways to prove this weak bound, and both are
instructive for our proof. To describe these, we define 
\[
\Psi\left(\ell,G\right)=\left\{ e\left(H\right):H\text{ is an }\ell\text{-vertex induced subgraph of }G\right\} .
\]
To prove that $\left|\Psi\left(G\right)\right|=\Omega\left(n^{3/2}\right)$,
it suffices to prove that $\left|\Psi\left(\ell,G\right)\right|=\Omega\left(\sqrt{n}\right)$
for each of $\Omega\left(n\right)$ different choices of $\ell$.

One way to do this, described by Alon and Kostochka~\cite{AK09},
is to use a discrepancy theorem and a switching argument. Erd\H os,
Goldberg, Pach and Spencer~\cite{EGPS88} proved that in any $n$-vertex
graph $G$ with density bounded away from 0 and 1, and any $\alpha\in\left(0,1\right)$
bounded away from 0 and 1, there are two induced subgraphs $G\left[W^{-}\right]$
and $G\left[W^{+}\right]$, with $\left|W^{-}\right|=\left|W^{+}\right|=\alpha n$,
such that $e\left(W^{+}\right)-e\left(W^{-}\right)=\Omega\left(n^{3/2}\right)$.
Recalling the Erd\H os--Szemer\'edi theorem that $O\left(1\right)$-Ramsey
graphs have density bounded away from 0 and 1, we can find such $W^{-}$
and $W^{+}$ in any $n$-vertex $O\left(1\right)$-Ramsey graph $G$.
One can then obtain a sequence of induced subgraphs $G\left[W_{0}\right],\dots,G\left[W_{\alpha n}\right]$
by starting with $W_{0}=W^{-}$ and switching vertices one-by-one
from $W^{-}$ into $W^{+}$. Formally, fix an ordering $w_{1}^{-},\dots,w_{\alpha n}^{-}$
of $W^{-}$ and an ordering $w_{1}^{+},\dots,w_{\alpha n}^{+}$ of
$W^{+}$ and let 
\[
W_{i}=\left\{ w_{1}^{-},\dots,w_{\alpha n-i}^{-}\right\} \cup\left\{ w_{1}^{+},\dots,w_{i}^{+}\right\} .
\]
Then, we have $\left|e\left(G\left[W_{i}\right]\right)-e\left(G\left[W_{i-1}\right]\right)\right|=\left|d_{W_{i}}\left(w_{i}^{+}\right)-d_{W_{i-1}}\left(w_{\alpha n-i+1}^{-}\right)\right|\le\alpha n$,
so as $e\left(G\left[W_{i}\right]\right)$ varies over an interval
of length $\Omega\left(n^{3/2}\right)$, it must attain $\Omega\left(\sqrt{n}\right)$
different values. This proves $\left|\Psi\left(\alpha n,G\right)\right|=\Omega\left(\sqrt{n}\right)$,
and we can apply this fact for $\Omega\left(n\right)$ different choices
of $\alpha=\ell/n$, proving that $\left|\Psi\left(G\right)\right|=\Omega\left(n^{3/2}\right)$.
We remark that this basic approach was refined by Alon and Kostochka
\cite{AK09} and by Alon, Balogh, Kostochka and Samotij~\cite{ABKS09},
to prove stronger bounds.

A second completely different way to prove that $\left|\Psi\left(G\right)\right|=\Omega\left(n^{3/2}\right)$,
due to Bukh and Sudakov~\cite[Proposition~3.1]{BS07}, is to make
use of the fact that Ramsey graphs have induced subgraphs with many
distinct degrees. Specifically, what Bukh and Sudakov proved was that
in any $O\left(1\right)$-Ramsey graph, there is an induced subgraph
with $\Omega\left(n\right)$ vertices which is \emph{diverse }in the
sense that most pairs of vertices have very different neighbourhoods
(to be precise, the symmetric difference of their neighbourhoods has
size $\Omega\left(n\right)$). In an $n'$-vertex diverse graph (with
$n'=\Omega\left(n\right)$), consider a random subset $U$ of $\alpha n'$
vertices (with $\alpha\in\left(0,1\right)$ bounded away from 0 and
1). By the diversity assumption, for most pairs of vertices $u,v$
their degrees $d_{U}\left(u\right)$, $d_{U}\left(v\right)$ into
$U$ are not too strongly correlated, and the probability they are
exactly equal turns out to be $O\left(1/\sqrt{n}\right)$. (A simple
intuitive reason for this probability is that $d_{U}\left(u\right)-d_{U}\left(v\right)$
is approximately normally distributed with standard deviation $\Theta\left(\sqrt{n}\right)$).
A simple linearity-of-expectation argument then shows that there is
an outcome of $G\left[U\right]$ with $\Omega\left(\sqrt{n}\right)$
different degrees. Finally, given an $\alpha n'$-vertex graph with
$\Omega\left(\sqrt{n}\right)$ different degrees, we can obtain $\left(\alpha n'-1\right)$-vertex
graphs with $\Omega\left(\sqrt{n}\right)$ different numbers of edges,
simply by choosing different vertices to delete. This proves that
$\left|\Psi\left(\alpha n'-1,G\right)\right|=\Omega\left(\sqrt{n}\right)$,
and again applying this fact for $\Omega\left(n\right)$ different
choices of $\alpha=\ell/n'$, it follows that $\left|\Psi\left(G\right)\right|=\Omega\left(n^{3/2}\right)$.

Observe that both the approaches described above seem to be somewhat
complementary. The discrepancy/switching argument, in its most basic
form, gives $\Omega\left(\sqrt{n}\right)$ different values of $e\left(G\left[U\right]\right)$
that are distributed fairly evenly over a range of length $\Omega\left(n^{3/2}\right)$.
On the other hand, the diversity/anticoncentration argument gives
$\Omega\left(\sqrt{n}\right)$ values of $e\left(G\left[U\right]\right)$
contained in an interval of length $O\left(n\right)$. It is natural
to try to combine both types of arguments to obtain better bounds.

In fact, recent developments bounding $\left|\Phi\left(G\right)\right|$
due to Narayanan, Sahasrabudhe and Tomon~\cite{NST16}, and ourselves
\cite{KS}, make this idea seem even more promising. In~\cite{NST16},
the authors made the simple observation (using the pigeonhole principle)
that in any $n$-vertex graph $G$, there is a set $A$ of $\sqrt{n}$
vertices with degrees lying in an interval of length $\sqrt{n}$.
If $G$ is diverse, and $U$ is a random vertex set of linear size,
then the degrees $d_{U}\left(x\right)$, for $x\in A$, are likely
to take $n^{1/2-o\left(1\right)}$ different values, very tightly
packed in an interval of length $O\left(\sqrt{n}\right)$. By augmenting
$U$ with different combinations of vertices in $A$, we can obtain
subgraphs of many different sizes, all lying in a fixed interval of
length $O\left(n\right)$. Adapting these ideas to our context, and
using the further refinements in~\cite{KS}, one can prove that we
can actually obtain $\Omega\left(n\right)$ values of $e\left(G\left[U\cup Y\right]\right)$
among subsets $Y\subseteq A$ of a certain fixed size, tightly packed
in an interval of length $O\left(n\right)$.

So, as a rough plan to prove \cref{conj:EFS}, one might start with
vertex subsets $W^{-},W^{+}$ of fixed size $\ell=\Theta\left(n\right)$
such that $e\left(W^{+}\right)-e\left(W^{-}\right)=\Omega\left(n^{3/2}\right)$,
provided by a discrepancy theorem. We would then switch between $W^{-}$
and $W^{+}$ to obtain subsets $W_{1},\dots,W_t$ such that among
the $e\left(W_{i}\right)$ there are $\Omega\left(\sqrt{n}\right)$
different values $e(W_{i_1}),e(W_{i_2}),\dots$ each separated by a distance of $\Omega\left(n\right)$.
One might then hope to somehow use diversity and anticoncentration
to show that each such $W_{i_j}$ has an ``augmenting set'' $A_{j}$
such that $e\left(W_{i_j}\cup Y\right)$ takes $\Omega\left(n\right)$
different values as $Y$ varies over subsets of $A_{j}$ with some fixed size $f\left(n\right)$.
We would moreover hope that for each $j$, the augmented
values $e\left(W_{i_j}\cup Y\right)$ fall in a specific interval of
length $O\left(n\right)$ that does not intersect the corresponding
interval for any other $j$. This would prove that $\left|\Psi\left(\ell+f\left(n\right),G\right)\right|=\Omega\left(n^{3/2}\right)$,
and this fact could be applied for $\Omega\left(n\right)$ different
choices of $\ell$ to prove that $\left|\Psi\left(G\right)\right|=\Omega\left(n^{5/2}\right)$.

There are several serious challenges associated with this kind of
approach. First, we need some way to introduce a random set $U$ of
linear size in order to use anticoncentration for our augmenting sets.
We have very little control over the number of edges in such a random
set (this number has variance $\Theta\left(n^{3}\right)$), so it
seems we must use the same random set for each $W_{i}$, and apply our switching argument \emph{after} our random set has been exposed. However, it seems that doing this would introduce new complications: the anticoncentration probabilities we are interested in are of order $O(1/\sqrt n)$, which is not small enough to apply the union bound over all $i$, given a single source of randomness. (It does not suffice to prove things for \emph{most} $i$, because the subsequence $(i_j)$ arising from the switching argument comprises a negligible fraction of all $i$).

Our approach
is to first prepare vertex sets $U^{0},W^{-},W^{+}$, each of a certain
linear size, such that 
\[
\left(e\left(W^{+}\right)+\alpha e\left(W^{+},U^{0}\right)\right)-\left(e\left(W^{-}\right)+\alpha e\left(W^{-},U^{0}\right)\right)=\Omega\left(n^{3/2}\right),
\]
for some $\alpha\in\left(0,1\right)$. Then, as above, we switch between
$W^{-}$ and $W^{+}$ to obtain a sequence of sets $W_{i}$, and identify
a well-separated subsequence of $\Omega\left(\sqrt{n}\right)$ sets
$W_{i_{j}}$ such that
\[
\left(e\left(W_{i_{j}}\right)+\alpha e\left(W_{i_{j}},U^{0}\right)\right)-\left(e\left(W_{i_{j-1}}\right)+\alpha e\left(W_{i_{j-1}},U^{0}\right)\right)=\Omega\left(n\right)
\]
for each $j$. Only then do we choose a random subset $U\subseteq U^{0}$
of size $\alpha\left|U^{0}\right|$, which we may use for anticoncentration.
By construction, the $e\left(W_{i_{j}}\cup U\right)$ are well-separated
in expectation, and the added randomness does not too severely disturb
the increments $e\left(W_{i}\cup U\right)-e\left(W_{i-1}\cup U\right)$.
Because we do not have any real control over the spacing
of the $i_{j}$, we must additionally carefully compensate for the buildup of deviations
caused by ``large gaps'' between the $i_{j}$.

Of course, before we even expose the random set $U$ we need to decide
which vertices should be in the augmenting sets $A_{j}$. Recall that
we would like to be able to use anticoncentration to obtain $\Omega\left(n\right)$
subgraph sizes of the form $e\left(W_{i_{j}}\cup U\cup Y\right)$,
for $Y\subseteq A_{j}$ of a fixed size. Provided that we have been
carefully maintaining appropriate diversity properties through the
construction, the only real requirement for this is that the $A_{j}$
are sufficiently large (of size at least $\Omega\left(\sqrt{n}\right)$).
However, ensuring that the different $A_{j}$ do not ``interfere''
with each other is a much more delicate task. With the pigeonhole
principle, for each $j$ we can show that there are $\sqrt{n}$ vertices
$v$ such that each $d_{W_{i_{j}}}\left(v\right)+\alpha d_{U^{0}}\left(v\right)$
is contained in an interval $I_{j}$ of length $\sqrt{n}$, and we
might hope to use such a set of vertices as our augmenting set $A_{j}$.
However, the pigeonhole principle gives us no guarantee of ``consistency''
between different $j$, and it might happen that the intervals $I_{j}$
jump around in such a way that there is a lot of overlap between the
augmented values $e\left(U\cup W_{i_{j}}\cup Y\right)$ for different
$j$. It seems to be quite difficult to carefully choose the $A_{j}$
in such a way that the $I_{j}$ are well-behaved.

Instead, we sidestep this issue, with the insight that it
is not actually necessary for all the vertices in $A_{j}$ to have
similar degrees into $W_{i_{j}}\cup U$; it suffices that $A_{j}$
has a large hypergraph matching $M_{j}\subseteq\binom{A_{j}}{k}$,
such that the \emph{sums} $d_{W_{i_{j}}\cup U}\left(\boldsymbol{v}\right)=d_{W_{i_{j}}\cup U}\left(v_{1}\right)+\dots+d_{W_{i_{j}}\cup U}\left(v_{k}\right)$
are similar for each $\boldsymbol{v}=\left\{ v_{1},\dots,v_{k}\right\} \in M_{j}$.
We may treat the edges of $M_{j}$ as we would treat single vertices,
forming our augmented values $e\left(U\cup W_{i_{j}}\cup\bigcup_{\boldsymbol{v}\in Z}\boldsymbol{v}\right)$
from subsets $Z\subseteq M_{j}$.

Being able to use a hypergraph matching instead of a set of vertices
affords us a lot of flexibility. For fixed $K\in\NN$, there are $\Omega\left(n^{K}\right)$
sets of $K$ vertices, and by the pigeonhole principle, $\Omega\left(n^{K-3}\right)$
of these $K$-sets $\boldsymbol{v}$ have exactly the same values
of $d_{W^{-}}\left(\boldsymbol{v}\right)$, the same values of $d_{W^{+}}\left(\boldsymbol{v}\right)$
and the same values of $d_{U^{0}}\left(\boldsymbol{v}\right)$. If
we obtain the $W_{i}$ by switching \emph{randomly} between $W^{-}$
and $W^{+}$, then we can show that the degrees $d_{W_{i}}\left(\boldsymbol{v}\right)$
are concentrated around a certain convex combination of $d_{W^{-}}\left(\boldsymbol{v}\right)$
and $d_{W^{+}}\left(\boldsymbol{v}\right)$. In this way we can produce
a collection of $K$-sets $\boldsymbol{v}$ such that the degrees
$d_{W_{i_{j}}\cup U}\left(\boldsymbol{v}\right)$ are quite well-behaved.

Of course, these $K$-sets are not disjoint, but for large $K$ we
may apply a weak form of the \emph{sunflower lemma} of Erd\H os and
Rado, to produce a hypergraph matching $M\subseteq\binom{V}{k}$ (with
$k\le K$) of almost linear size, which has similarly well-behaved
degrees. With this as a starting point, it becomes feasible to use
the pigeonhole principle to obtain appropriate sub-matchings $M_{j}'\subseteq M$,
and modulo a lot of technical details we are able to more or less
implement the plan described above. To summarise, for each of $\sqrt{n}$
choices of $j$ we use anticoncentration and a generalised notion
of diversity to produce $\Omega\left(n\right)$ values of $e\left(U\cup W_{i_{j}}\cup\bigcup_{\boldsymbol{v}\in Z}\boldsymbol{v}\right)$
among $Z\subseteq M_{j}'$ of a certain fixed size, in such a way
that there is little overlap between the values for different $j$.
This gives us $\Omega\left(n^{3/2}\right)$ subgraphs with the same
number of vertices and different numbers of edges, and varying the
size of $U$ allows us to prove that $\left|\Psi\left(G\right)\right|=\Omega\left(n^{5/2}\right)$,
as desired.

\section{Basic tools}
\label{sec:basic-tools}

\subsection{Diverse neighbourhoods in Ramsey graphs}\label{subsec:diversity}

In~\cite{BS07} Bukh and Sudakov introduced the notion of \emph{diversity}:
an $n$-vertex graph is said to be diverse if $\left|N\left(x\right)\triangle N\left(y\right)\right|=\Omega\left(n\right)$
for most pairs of distinct vertices $x,y$. We will need a slightly
stronger notion than diversity, which we introduced in~\cite{KS}.
Say an $n$-vertex graph is \emph{$\left(\delta,\varepsilon\right)$-rich}
if for any vertex subset $W$ with $\left|W\right|\ge\delta n$, at
most $n^{1/5}$ vertices $v$ have $\left|N\left(v\right)\cap W\right|<\varepsilon\left|W\right|$
or $\left|\overline{N\left(v\right)}\cap W\right|<\varepsilon\left|W\right|$. Note that a graph which is $(\delta,\varepsilon)$-rich is also $(\delta',\varepsilon)$-rich, if $\delta'>\delta$.  We remark that a slightly different definition of richness appeared in the published version of this paper, which was not quite suitable for our application. We thank Mantas Baksys and Xuanang Chen for bringing this to our attention.
The next lemma appears as~\cite[Lemma~4]{KS}, showing that Ramsey
graphs contain large rich induced subgraphs.
\begin{lem}
\label{lem:rich}For any $C,\delta>0$, there exist $\varepsilon=\varepsilon\left(C\right)>0$
and $c=c\left(C,\delta\right)>0$ and $n_0=n_0(\delta)$ such that if $n\ge n_0$ then every $n$-vertex $C$-Ramsey
graph contains a $\left(\delta,\varepsilon\right)$-rich induced subgraph
on at least $cn$ vertices.
\end{lem}

In~\cite{KS}, the reason we introduced $\left(\delta,\varepsilon\right)$-richness
was to derive a type of diversity for pairs of vertices. Here we will
need a type of diversity for larger sets of vertices. (recall from
\cref{subsec:definitions} the non-standard multiset definitions of
$N\left(\boldsymbol{x}\right),N\left(\boldsymbol{y}\right)$ and $N\left(\boldsymbol{x}\right)\triangle N\left(\boldsymbol{y}\right)$).
\begin{lem}
\label{lem:diversity-tuples}Fix $k\in\NN$ and let $G$ be a $\left(\delta,\varepsilon\right)$-rich
graph on an $n$-vertex set $V$. Then, for each $\boldsymbol{x}\in\binom{V}{k}$
with $\left|\bigcap_{v\in\boldsymbol{x}}N\left(v\right)\right|\ge\delta n$,
one cannot find a collection of $n^{1/5}$ vertex subsets $\boldsymbol{y}\in\binom{V}{k}$
(disjoint from $\boldsymbol{x}$ and each other) such that $\left|N\left(\boldsymbol{x}\right)\triangle N\left(\boldsymbol{y}\right)\right|<\delta\varepsilon n$.
\end{lem}

\begin{proof}
Let $W=\bigcap_{v\in\boldsymbol{x}}N\left(v\right)$. Suppose the
statement of the lemma were false, and such a collection $Y$ of vertex
subsets existed. Then, for each $y\in\boldsymbol{y}$, for $\boldsymbol{y}\in Y$,
we would have $\left|\overline{N\left(y\right)}\cap W\right|\le\left|N\left(\boldsymbol{x}\right)\triangle N\left(\boldsymbol{y}\right)\right|<\varepsilon\left|W\right|$,
and the set of all such $y$ would contradict $\left(\delta,\varepsilon\right)$-richness.
\end{proof}
\cref{lem:diversity-tuples} only applies to $\boldsymbol{x}\in\binom{V}{k}$
such that $\bigcap_{v\in\boldsymbol{x}}N\left(v\right)$ is large.
In order to apply it, we next show that in a rich graph, $\bigcap_{v\in\boldsymbol{x}}N\left(v\right)$
is large for almost all $\boldsymbol{x}\in\binom{V}{k}$.
\begin{lem}
\label{lem:few-bad}Fix $k\in\NN$ and let $G$ be a $\left(\delta,\varepsilon\right)$-rich
graph on an $n$-vertex set $V$, for $\delta\le\varepsilon^{k-1}$.
Then there are at most $n^{k-1+1/5}$ subsets $\boldsymbol{v}\in\binom{V}{k}$
such that $\left|\bigcap_{v\in\boldsymbol{v}}N\left(v\right)\right|<\varepsilon^{k}n$.
\end{lem}

\begin{proof}
We will prove by induction that there are at most $qn^{q-1+1/5}$
``bad'' ordered $q$-tuples $\boldsymbol{v}\in V^{q}$ such that
$\left|\bigcap_{v\in\boldsymbol{v}}N\left(v\right)\right|<\varepsilon^{q}n$,
for all $1\le q\le k$. This will prove that there are at most $kn^{k-1+1/5}/k!\le n^{k-1+1/5}$
subsets $\boldsymbol{v}\in\binom{V}{k}$ such that $\left|\bigcap_{v\in\boldsymbol{v}}N\left(v\right)\right|<\varepsilon^{k}n$.

First note that the base case $q=1$ follows directly from $\left(\delta,\varepsilon\right)$-richness,
with $W=V$. Then, assume for induction that our desired bound holds
for $q-1$; we will prove it for $q$. First, there are at most $\left(q-1\right)n^{q-1+1/5}$
bad $q$-tuples obtained by appending a vertex to a bad $\left(q-1\right)$-tuple.
Then, for each $\left(q-1\right)$-tuple $\boldsymbol{v}$ which is
not bad (meaning $\left|\bigcap_{v\in\boldsymbol{v}}N\left(v\right)\right|\ge\varepsilon^{q-1}n$),
by $\left(\delta,\varepsilon\right)$-richness there are
at most $n^{1/5}$ vertices $w$ with $\left|N\left(w\right)\cap\bigcap_{v\in\boldsymbol{v}}N\left(v\right)\right|<\varepsilon\left|\bigcap_{v\in\boldsymbol{v}}N\left(v\right)\right|$,
meaning that there are at most $n^{q-1+1/5}$ bad-$q$-tuples that
can be obtained by appending a vertex to a not-bad $\left(q-1\right)$-tuple,
and at most $qn^{q-1+1/5}$ bad $q$-tuples total.
\end{proof}

\subsection{Tools from extremal (hyper)graph theory}

We will make frequent use of Tur\'an's theorem to find large independent
sets in various auxiliary graphs. The following form of the theorem
appears, for example, in~\cite{AS}.
\begin{prop}
\label{prop:turan}Every $n$-vertex graph $G$ contains an independent
set of size at least
\[
\sum_{v\in V\left(G\right)}\frac{1}{d\left(v\right)+1}\ge\frac{n^{2}}{\sum_{v\in V\left(G\right)}\left(d\left(v\right)+1\right)}=\Omega\left(\min\left\{ n,\frac{n^{2}}{e\left(G\right)}\right\} \right).
\]
\end{prop}

Next, a \emph{sunflower} in a hypergraph is a subgraph in which every
pair of edges has the same intersection (this common intersection
is called the \emph{kernel}, and removing the kernel from each edge
gives the \emph{petals}). We will need the following weak form of the Erd\H os--Rado
sunflower lemma~\cite{ER61}, which one can easily prove by induction on the uniformity of a hypergraph.
\begin{lem}
\label{lem:sunflower}Fix $k\in\NN$ and let $H$ be a $k$-uniform
hypergraph with $m$ edges. Then $H$ contains an $\Omega\left(m^{1/k}\right)$-edge
sunflower.
\end{lem}

\subsection{Probabilistic tools}\label{subsec:prob-tools}

We will need concentration and anticoncentration inequalities for
random variables arising from random subsets of given sizes. Say a
random variable $X$ is of \emph{$\left(n,p,b\right)$-hypergeometric}
type if it can be expressed in the form $X=\sum_{i\in I}a_{i}$,
where $a_{1},\dots,a_{n}\in\RR$ are fixed, $\left|a_{i}\right|\le b$
for each $i$, and $I$ is a uniformly random subset of $\left\{ 1,\dots,n\right\} $
of size $pn$. The following concentration lemma follows directly
from~\cite[Corollary~2.2]{GIKM17}.
\begin{lem}
\label{lem:concentration}Suppose $X$ is of $\left(n,p,b\right)$-hypergeometric
type. Then, for any $t\in\RR$,
\[
\Pr\left(\left|X-\E X\right|\ge t\right)=\exp\left(-\Omega\left(\frac{t^{2}}{nb^{2}\min\left\{ p,1-p\right\} }\right)\right).
\]
\end{lem}

Next, say that $X$ as above (of \emph{$\left(n,p,b\right)$}-hypergeometric
type) is of \emph{$\left(n,p,b,r\right)^{*}$}-hypergeometric type if
moreover $\left|a_{i}\right|\ge1/b$ for at least $r$ indices $i$ (that is, many $a_{i}$ are
bounded away from zero as well as being bounded in size). The following
central limit theorem directly follows from a classical quantitative
central limit theorem first proved by Bikelis \cite{Bik69} (see also
\cite{Hog78}).
\begin{lem}
\label{lem:Bikelis}Fix $b>0$ and suppose $X$ is of $\left(n,p,b,n/b\right)^{*}$-hypergeometric
type, with $|\E X|\le n/(2b^2)$. Let $F$ be the distribution function of $\left(X-\E X\right)/\sqrt{\Var X}$
and let $G$ be the standard Gaussian distribution function. Then
for all $z\in\RR$,
\[
\left|F\left(z\right)-G\left(z\right)\right|=O\left(\frac{1}{\sqrt{p(1-p)n}}\right).
\]
\end{lem}

We only need \cref{lem:Bikelis} for anticoncentration,
so we state a simple corollary for later use.
\begin{lem}
\label{lem:anticoncentration}Suppose $X$ is of $\left(n,p,O(1),\Omega(n)\right)^{*}$-hypergeometric type.
Then, for any $-\sqrt n<x<\sqrt n$,
\[
\Pr\left(X=x\right)=O\left(\frac{1}{\sqrt{p(1-p)n}}\right).
\]
\end{lem}
\begin{proof}
If say $|\E X|\le n^{2/3}$ then the desired result follows from \cref{lem:Bikelis}. Otherwise, the desired result follows from \cref{lem:concentration}, since with probability $1-e^{-\Omega(n^{1/3})}$, $X$ does not even fall in the interval between $-\sqrt n$ and $\sqrt n$.
\end{proof}

We also make the following simple observation, which will be convenient
to show that various discrepancy properties we are able to establish
will persist with positive probability through certain kinds of random
sampling. If a random variable is of $\left(n,1/2,b\right)$-hypergeometric
type for some $n$ and $b$, say it is of $\left(1/2\right)$-hypergeometric
type.
\begin{lem}
\label{lem:symmetric}Suppose $X$ is of $\left(1/2\right)$-hypergeometric
type. Then, $X-\E X$ has the same distribution as $\E X-X$, and
in particular, $X\ge\E X$ with probability at least $1/2$.
\end{lem}

\begin{proof}
Suppose that $X=\sum_{i\in I}a_{i}$, and let $X=\sum_{i\notin I}a_{i}$.
Since $I$ is a random subset of exactly half the indices $\left\{ 1,\dots,n\right\} $,
it has the same distribution as its complement $\overline{I}$, so
$X$ has the same distribution as $X'$. But observe that 
\[
\left(X+X'\right)/2=\sum_{i=1}^{n}a_{i}/2=\E X=\E X',
\]
so $\E X-X=X'-\E X'$.
\end{proof}
We remark that \cref{lem:concentration} (respectively \cref{lem:symmetric}) trivially remains true when the relevant random variables are translated by a fixed constant. We will therefore frequently abuse notation and say that translations of random variables of $(n,p,b)$-hypergeometric type (respectively $(1/2)$-hypergeometric type) are themselves  of $(n,p,b)$-hypergeometric type (respectively $(1/2)$-hypergeometric type).

Throughout the proof we will also frequently use Markov's inequality;
the statement and proof can be found, for example, in~\cite{AS}.

\subsection{Switching analysis}

In this subsection we collect some simple lemmas that will be useful
for tracking how certain parameters change as we gradually switch
from one vertex subset to another. First, we show that if we move
between two distant values, and most of the incremental steps are
not too extreme, then there are many intermediate steps with ``well-separated''
values.
\begin{lem}
\label{lem:well-separated}Consider a sequence $p_{0},\dots,p_{\tau}$
with $p_{\tau}-p_{0}\ge\lambda$. Let $\Delta_{i}=p_{i}-p_{i-1}$
and suppose that for some $\rho$ we have
\[
\sum_{i:\Delta_{i}>\rho}\Delta_{i}\le\kappa.
\]
Then, for any $\sigma\le\rho$ there is an increasing subsequence
$0=i_{1},\dots,i_{s}=\tau$, with $s\ge\lambda/\left(\rho+\sigma\right)-\kappa/\rho$,
such that $p_{i_{j}}-p_{i_{j-1}}\ge\sigma$ for all $1\le j\le s$.
\end{lem}

\begin{proof}
We view $p_{i}$ as the position of a ``particle'' at ``time''
$i$. In the interval from $p_{0}$ to $p_{\tau}$ , consider $\lambda/\left(\rho+\sigma\right)$
sub-intervals of length $\rho$ separated by a distance of at least
$\sigma$, with the first sub-interval containing $p_{0}$ and the
last containing $p_{\tau}$. We say a sub-interval $I$ is ``further''
than a sub-interval $I'$ if $I$ is closer to $p_{\tau}$ than $I'$.

Let $i_{1}=0$ and let $I_{1}$ be the sub-interval containing $p_{0}$.
For $j>1$ let $i_{j}>i_{j-1}$ be the first time $i$ that $p_{i}$
is in a sub-interval further than $I_{j-1}$ and let $I_{j}$ be this
sub-interval. This process terminates when there is no sub-interval
further than $I_{j}$ (let $s=j$ for this value of $j$, and redefine
$i_{s}=\tau$). Observe that at most $\kappa/\rho$ intervals were
skipped, so $s\ge\lambda/\left(\rho+\sigma\right)-\kappa/\rho$.
\end{proof}
Next, the following lemma shows that if an ensemble of values move
slowly in a bounded region, then at least one value ``follows the
crowd'' for quite a long time.
\begin{lem}
\label{lem:lonely}Consider an interval $I\subseteq\ZZ$ with $\left|I\right|=\lambda$,
and consider a ``time horizon'' $\tau\in\NN$. Consider a set of
``particles'' $R$, and for each $a\in R$ let $p_{i}\left(a\right)\in I$
represent the ``position'' of $a$ at time $i$, in such a way that
$\left|p_{i}\left(a\right)-p_{i-1}\left(a\right)\right|\le\rho$ for
each $0<i\le\tau$ (that is, the particles move with ``speed'' at
most $\rho$). For $\sigma,\mu>0$, say a particle $a$ is \emph{lonely}
at time $0\le i\le\tau$ if 
\[
\left|\left\{ b\in R:\left|p_{i}\left(b\right)-p_{i}\left(a\right)\right|\le\sigma\right\} \right|<\mu.
\]
(That is, a particle is lonely if there are few other particles close
to it). Now, if $\tau\le\left|R\right|\sigma^{2}/\left(8\mu\rho\lambda\right)$
then there is a particle $a$ which is never lonely.
\end{lem}

\begin{proof}
Say a particle $a$ is \emph{crowded }if 
\[
\left|\left\{ b\in R:\left|p_{i}\left(b\right)-p_{i}\left(a\right)\right|\le\sigma/2\right\} \right|\ge\mu.
\]
At any time $i$, fewer than $2\mu\lambda/\sigma$ particles are not
crowded. To see this, divide $I$ into $2\lambda/\sigma$ sub-intervals
of length $\sigma/2$. If a sub-interval contains at least $\mu$
particles then all particles in that sub-interval are crowded.

Now, if a particle is crowded for every time $j\sigma/\left(4\rho\right)$
(among $j\in\NN$ with $j\le4\rho\tau/\sigma$), then it is never
lonely. To see this, observe that it takes at least $\sigma/\left(4\rho\right)$
time steps for a crowded particle to become lonely. This is because
the separation between that particle and the particles within distance
$\sigma/2$ must increase by $\sigma/2$, and if two particles are
moving away from each other their separation increases by at most
$2\rho$ per time step. By this fact and the preceding paragraph,
there are fewer than $\left(4\rho\tau/\sigma\right)\left(2\mu\lambda/\sigma\right)$
particles that are ever lonely, and if $8\rho\tau\mu\lambda/\sigma^{2}\le\left|R\right|$
then there is a particle that is never lonely.
\end{proof}

\section{Proof of \texorpdfstring{\cref{conj:EFS}}{Theorem~\ref{conj:EFS}}}\label{sec:EFS-proof}

As in \cref{sec:discussion}, define 
\[
\Psi\left(\ell,G\right)=\left\{ e\left(H\right):H\text{ is an }\ell\text{-vertex induced subgraph of }G\right\} .
\]
As discussed in \cref{sec:discussion}, in the previous bounds on $\left|\Psi\left(G\right)\right|$ in
\cite{BS07,AK09,ABKS09}, the approach was to show that $\Psi\left(\ell,G\right)$
is large for each of $\Omega\left(n\right)$ specific choices of $\ell$.
In this paper it will be convenient to have slightly more flexibility:
we show that $\Psi\left(\ell',G\right)$ is large for $\Omega\left(n\right)$
different choices of $\ell'$, but we do not specify precisely which
choices they are. We will prove the following lemma, which suffices
to prove \cref{conj:EFS}.
\begin{lem}
\label{lem:per-l}For any fixed $C$, there is $c>0$ such that the
following holds. For any $n$-vertex $C$-Ramsey graph $G$, there
are $f,h\in\NN$ such that for any $cn\le\ell\le2cn$, either $\left|\Psi\left((\ell-f)+h,G\right)\right|=\Omega\left(n^{3/2}\right)$
or $\left|\Psi\left(2(\ell-f)+h,G\right)\right|=\Omega\left(n^{3/2}\right)$.
\end{lem}

\begin{proof}[Proof of \cref{conj:EFS} given \cref{lem:per-l}]
We have
\[
\left|\Psi\left(G\right)\right|\ge\frac{1}{2}\sum_{\ell=cn}^{2cn}\left(\left|\Psi\left((\ell-f)+h,G\right)\right|+\left|\Psi\left(2(\ell-f)+h,G\right)\right|\right)=\Omega\left(n^{5/2}\right).\qedhere
\]
\end{proof}
The first ingredient for the proof of \cref{lem:per-l} will be the
following lemma asserting the existence of a collection of vertex
sets with certain discrepancy, regularity and diversity properties.
\begin{lem}
\label{lem:step-2}For any fixed $C$, there are $K\in\NN$ and
$c>0$ such that the following holds. For any $n$-vertex $C$-Ramsey
graph $G$, any $\alpha=\alpha(n)\ge 1/2$ and any $cn\le\ell\le2cn$, there are disjoint vertex sets
$W^{-},W^{+},U^{0},A$, and a $k$-uniform hypergraph perfect matching
$M\subseteq\binom{A}{k}$ of $A$ for some $k\le K$, satisfying the
following properties.
\begin{enumerate}
\item $\left|W^{-}\right|=\left|W^{+}\right|=cn$, $\left|A\right|=\Omega\left(n^{3/4}\right)$,
and either $\left|U^{0}\right|=\ell$ or $\left|U^{0}\right|=2\ell$;
\item $\left(e\left(W^{+}\right)+\alpha e\left(U^{0},W^{+}\right)\right)-\left(e\left(W^{-}\right)+\alpha e\left(U^{0},W^{-}\right)\right)=\Omega\left(n^{3/2}\right)$;
\item there are $d_{W^{-}},d_{W^{+}},d_{U^{0}}\in\NN$ such that $d_{W^{-}}\left(\boldsymbol{v}\right)=d_{W^{-}}$,
$d_{W^{+}}\left(\boldsymbol{v}\right)=d_{W^{+}}$ and $d_{U^{0}}\left(\boldsymbol{v}\right)=d_{U^{0}}$
for all $\boldsymbol{v}\in M$;
\item for each $\left\{ \boldsymbol{x},\boldsymbol{y}\right\} \in\binom{M}{2}$
we have $\left|N_{U^{0}}\left(\boldsymbol{x}\right)\triangle N_{U^{0}}\left(\boldsymbol{y}\right)\right|=\Omega\left(n\right)$.
\end{enumerate}
(Here, the implied constants in all asymptotic notation depend on $C$ but not $\alpha$).
\end{lem}

We will prove \cref{lem:step-2} in \cref{sec:step-2}. We remark that our proof can be easily modified to give $|A|=\Omega(n^{1-\eta})$ for any $\eta>0$, and all that we actually need for the proof of \cref{lem:per-l} is that $|A|=\Omega(n^{1/2+\eta})$ for some $\eta>0$. The choice of the exponent $3/4$ is merely for concreteness.

The next ingredient
is the following lemma, showing that with positive probability we
can augment a random set of vertices in many different ways to get
induced subgraphs with many different numbers of edges.
\begin{lem}
\label{lem:sub-random}Consider any $\n D=\n D\left(n\right)\in\NN$
with $\n D=\omega\left(\log n\right)$, and suppose in a graph $G$
we have disjoint vertex subsets $W,A,U^{0}$ and a hypergraph perfect
matching $M\subseteq\binom{A}{k}$ for some $k=O\left(1\right)$,
satisfying the following properties.
\begin{enumerate}
\item $\left|U^{0}\right|\ge3\n D$, and $\left|M\right|=\Omega\left(\sqrt{\n D}\right)$;
\item $\left|N_{U^{0}}\left(\boldsymbol{x}\right)\triangle N_{U^{0}}\left(\boldsymbol{y}\right)\right|=\Omega\left(\left|U^{0}\right|\right)$
for each $\left\{ \boldsymbol{x},\boldsymbol{y}\right\} \in\binom{M}{2}$;
\item there are $d_{W},d_{U^{0}}\in\NN$ such that $d_{U^{0}}\left(\boldsymbol{v}\right)=d_{U^{0}}$
and $d_{W}\left(\boldsymbol{v}\right)=d_{W}+o\left(\sqrt{\n D}\right)$
for all $\boldsymbol{v}\in M$.
\end{enumerate}
Then, there are $B=O\left(1\right)$ and $\delta=\Omega\left(1\right)$
(depending on the implied constants in the above asymptotic notation,
but not depending on $\n D$) such that the following holds. Consider
any $\n Z\le\delta\sqrt{\n D}$, let $D$ be a uniformly random subset
of $\n D$ elements of $U^{0}$, let $U=U^{0}\setminus D$ and define $\alpha$ to satisfy
$\n D=(1-\alpha)|U^0|$.
With probability at least $1/4$,
\[
\left|\left\{ e\left(W\cup U\cup\bigcup_{\boldsymbol{z}\in Z}\boldsymbol{z}\right):\;Z\subseteq M,\;\left|Z\right|=\n Z,\;\left|e\left(U,\bigcup_{\boldsymbol{z}\in Z}\boldsymbol{z}\right)-\alpha\n Zd_{U^{0}}\right|\le B\n D\right\} \right|=\Omega\left(\n Z\sqrt{\n D}\right).
\]
\end{lem}

We will prove \cref{lem:sub-random} in \cref{sec:sub-random}, using some ideas from~\cite{KS, NST16}. To interpret its conclusion in words, it says that one can obtain $\Omega\left(\n Z\sqrt{\n D}\right)$ induced subgraphs with different numbers of edges, by augmenting $W\cup U$ with different subsets $Z\subseteq M$ of size $\n Z$. Moreover, this is still true if we restrict our attention to those subsets $Z$ such that there are about the expected number of edges $\alpha \n Z d_{U^0}$ between $U$ and $Z$.

Finally, we show how to combine \cref{lem:step-2} and \cref{lem:sub-random}
to prove \cref{lem:per-l}.
\begin{proof}[Proof of \cref{lem:per-l}]
Apply \cref{lem:step-2} with $\alpha=\left(\ell-c'n\right)/\ell$, for some small $c'$ (depending on $c$) that will be chosen later to satisfy certain inequalities. Until we finally determine the value of $c'$, the constants
implied by all asymptotic notation in this section will be independent
of $c'$ (that is, if say $f\le c'n$, we may write $f=O(c'n)$ but not $f=O(n)$). Choose $\n D$ to satisfy $\n D=(1-\alpha)|U^0|$ (so $\n D=c'n$ or $\n D=2c'n$, and in particular $\n D\le 2c'n$). Let $\n W=cn$ and consider uniformly random orderings $w_{1}^{-},\dots,w_{\n W}^{-}$
of $W^{-}$ and $w_{1}^{+},\dots,w_{\n W}^{+}$ of $W^{+}$. For $0\le i\le\n W$
let
\[
W_{i}^{-}=\left\{ w_{1}^{-},\dots,w_{\n W-i}^{-}\right\} ,\quad W_{i}^{+}=\left\{
\vphantom{w_{\n W-i}^{-}}
w_{1}^{+},\dots,w_{i}^{+}\right\} ,\quad W_{i}=W_{i}^{-}\cup W_{i}^{+}.
\]
This means each individual $W_{i}^{-}$ (respectively $W_{i}^{+}$)
is a uniformly random subset of $\n W - i$ elements of $W^{-}$ (respectively,
$i$ elements of $W^{+}$). Define
\[
d_{W_{i}^{-}}=\frac{\n W-i}{\n W}d_{W^{-}},\quad d_{W_{i}^{+}}=\frac{i}{\n W}d_{W^{+}},\quad d_{W_{i}}=d_{W_{i}^{-}}+d_{W_{i}^{+}}.
\]

Now, for each $0\le i\le\n W$ and $\boldsymbol{v}\in M$, the random
variable $d_{W_{i}^{-}}\left(\boldsymbol{v}\right)$ (respectively
$d_{W_{i}^{+}}\left(\boldsymbol{v}\right)$) is of $\left(\n W,p,O\left(1\right)\right)$-hypergeometric
type, for $p=\left(\n W-i\right)/\n W$ (respectively, for $p=i/\n W$),
and has mean $d_{W_{i}^{-}}$ (respectively, mean $d_{W_{i}^{+}}$).
By \cref{lem:concentration} (with $t=\sqrt n \log n$) and the union bound, we can fix an outcome
of the orderings $w_{1}^{-},\dots,w_{\n W}^{-}$ and $w_{1}^{+},\dots,w_{\n W}^{+}$
such that $
\left|d_{W_{i}}\left(\boldsymbol{v}\right)-d_{W_{i}}\right|\le\sqrt{n}\log n
$
for each $0\le i\le\n W$ and $\boldsymbol{v}\in M$. (Note that if $p=o(1)$ the estimate in \cref{lem:concentration} only becomes stronger).

This concentration would suffice to prove an approximate version of
\cref{conj:EFS}, that $|\Psi(G)| = n^{5/2}/\log^{O(1)} n$ (simply using the single matching $M$ to augment each $W_i$). However, in order to obtain an exact result we need to
eliminate the logarithmic factor in the estimate for $\left|d_{W_{i}}\left(\boldsymbol{v}\right)-d_{W_{i}}\right|$. We have the freedom to do this because $M$ (coming from \cref{lem:step-2}, of size $\Omega(n^{3/4})$) is much larger than the necessary size $\Theta(\sqrt{n})$ of our ``augmenting sets'' (as outlined in \cref{sec:discussion}). In fact, we could use the pigeonhole principle to easily show that for each $i$ there is a subset $M_i\subseteq M$ of size $\Theta(\sqrt{n})$ such that the degrees $d_{W_i}(\boldsymbol v)$, for $\boldsymbol v\in M_i$, are contained in a tiny interval of length only $O(n^{1/4}\log n)$, centered at some point $d_i$. But because we require consistency between the $i$ (in particular, we do not want the $d_i$ to vary too much), things are a bit more delicate, and we will apply \cref{lem:lonely}.
\begin{claim}
\label{claim:lonely}There are $d_{0},\dots,d_{\n W}\in\NN$ and $M_{1},\dots,M_{\n W}\subseteq M$
such that the following hold.
\begin{enumerate}
\item [(i)]Each $\left|M_{i}\right|\ge\sqrt{n}$;
\item [(ii)]For each $0\le i\le\n W$ and each $\boldsymbol{v}\in M_{i}$,
we have $\left|d_{W_{i}}\left(\boldsymbol{v}\right)-d_{i}\right|=o\left(\sqrt{n}\right)$;
\item [(iii)]For each $0<i\le\n W$ we have $\left|d_{i}-d_{i-1}\right|=O\left(\sqrt{n}\log n\right)$,
and actually $\left|d_{i}-d_{i-1}\right|=o\left(\sqrt{n}\right)$
for all but $O\left(n^{1/4}\log^{3}n\right)$ indices $i$.
\end{enumerate}
\end{claim}

\begin{proof}
The indices $i$ will represent points in time. Let $\mu=\sqrt{n}$, let
$\sigma=\sqrt{n}/\log n=o\left(\sqrt{n}\right)$, let $\lambda=2\sqrt{n}\log n$
and let $I$ be the interval of integers between $-\lambda/2$ and
$\lambda/2$. Let $R=M$ (recalling that $\left|M\right|=\Omega\left(n^{3/4}\right)$)
and for each $0\le i\le\n W$ and $\boldsymbol{v}\in R$ let $p_{i}\left(\boldsymbol{v}\right)=d_{W_{i}}\left(\boldsymbol{v}\right)-d_{W_{i}}\in I$.

Note that each $\boldsymbol v\in M$ has size $k$, so for each $0<i\le\n W$, we have $\left|d_{\boldsymbol v}\left(\vphantom{w_{\n W-i+1}^-}w_i^+\right)-d_{\boldsymbol v}\left(w_{\n W-i+1}^-\right)\right|\le k$ and therefore $\left|d_{W_i}\left(\boldsymbol{v}\right)-d_{W_{i-1}}\left(\boldsymbol{v}\right)\right|\le k$. Also, we can compute
\begin{equation}
\left|d_{W_i}-d_{W_{i-1}}\right|=\left|\frac{d_{W^+}-d_{W^-}}{\n W}\right|\le k.\label{eq:ideal-slow-change}
\end{equation}
So, with $\rho=2k$, we have $\left|p_{i}\left(\boldsymbol{v}\right)-p_{i-1}\left(\boldsymbol{v}\right)\right|\le\rho$. Divide the range of ``times'' between $0$
and $\n W$ into $\n W/\tau$ sub-ranges of lengths $\tau=\left|R\right|\sigma^{2}/\left(8\rho\mu\lambda\right)=O(n^{3/4}/\log^3 n)$.
For each such sub-range $T$, by \cref{lem:lonely} there is some $\boldsymbol{v}_{T}\in R$
which is never lonely in that range; fix such a $\boldsymbol{v}_{T}$
and for each $i\in T$ let $d_{i}=d_{W_{i}}\left(\boldsymbol{v}_{T}\right)$.
For each $0\le i\le\n W$ let $M_{i}\subseteq M$ be a set of $\mu$
elements $\boldsymbol{v}\in M$ satisfying $\left|d_{W_{i}}\left(\boldsymbol{v}\right)-d_{i}\right|\le\sigma$, which exists by the definition of loneliness.
Recalling \cref{eq:ideal-slow-change}, observe that $\left|d_{i}-d_{i-1}\right|\le\lambda=O\left(\sqrt{n}\log n\right)$
for all $0<i\le\n W$. Moreover, for all $i$ except the $\n W/\tau=O\left(n^{1/4}\log^{3}n\right)$
times where there is a ``transition'' between sub-ranges, there is $\boldsymbol v$ such that $\left|d_{i}-d_{i-1}\right|=\left|d_{W_i}(\boldsymbol v)-d_{W_{i-1}}(\boldsymbol v)\right|\le k=o\left(\sqrt{n}\right)$.
\end{proof}
Next, (more or less) as described in \cref{sec:discussion}, we identify a subsequence of indices $i$ leading to subgraph sizes that are ``well-separated'' in a certain sense. Let $e_{i}=e\left(W_{i}\right)+\alpha e\left(U^{0},W_{i}\right)+\n Zd_{i}$,
where $\n Z=\delta\sqrt{c'n}/k\le \delta \sqrt{\n D}$ for some small $\delta=\delta\left(C\right)>0$
(not depending on $c'$) to be determined. The precise significance of these quantities $e_i$ will become clear later, but the rough idea (as sketched in \cref{sec:discussion}) is that we will eventually want to consider subgraphs consisting of some $W_i$, a random $\alpha$-proportion of the elements of $U^0$, and $\n Z$ vertices of $M_i$. Note that each $d_{i}$ was defined to be equal to some $d_{W_i}(\boldsymbol v)\le k n$, so $\left|d_{\n W}-d_0\right|\le kn=O(n)$,
and recall that $\n Z\le\sqrt {c'n}$. So, for small $c'$, by property 2 of \cref{lem:step-2}, we have
$$e_{\n W}-e_{0}=
\left(e\left(W^{+}\right)+\alpha e\left(U^{0},W^{+}\right)\right)-\left(e\left(W^{-}\right)+\alpha e\left(U^{0},W^{-}\right)\right)+\n Z (d_{\n W}-d_0)=\Omega\left(n^{3/2}\right).$$
Now, for $0<i\le\n W$ let $\Delta_{i}=e_{i}-e_{i-1}$. Observe that
\begin{align*}
&\left|\left(e\left(W_{i}\right)+\alpha e\left(U^{0},W_{i}\right)\right)-\left(e\left(W_{i-1}\right)+\alpha e\left(U^{0},W_{i-1}\right)\right)\right|\\
&\qquad = \left|\left(d_{W_i}\left(
\vphantom{w^-_{\n W-i}}
w^+_i\right)+\alpha d_{U^0}\left(
\vphantom{w^-_{\n W-i}}
w^+_i\right)\right)-\left(d_{W_{i-1}}\left(w^-_{\n W-i+1}\right)+\alpha d_{U^0}\left(w^-_{\n W-i+1}\right)\right)\right|
\le (1+\alpha)n\le 2n,
\end{align*}
so the only way to have $\Delta_{i}> 3n$ is if $\left|d_i-d_{i-1}\right|=\Omega(\sqrt n)$. By (iii) of \cref{claim:lonely},
\[
\sum_{i:\left|\Delta_{i}\right|>3n}\left|\Delta_{i}\right|=O\left(\n Z\left(n^{1/4}\log^{3}n\right)\left(\sqrt{n}\log n\right)\right)=o\left(n^{3/2}\right).
\]
By \cref{lem:well-separated} (with $\tau=\n W=\Theta(n)$, $\lambda=\Omega(n^{3/2})$, $\rho=3n$, $\kappa=o(n^{3/2})$ and $\sigma=n$) there is an increasing subsequence of
indices $0=i_{1},\dots,i_{t}=\n W$, with $t=\Omega\left(\sqrt{n}\right)$,
such that $e_{i_{j}}-e_{i_{j-1}}\ge n$ for each $1<j\le t$.

Now, let $D$ be a uniformly random subset of $\n D$ elements of
$U^{0}$, and let $U=U^{0}\setminus D$. For a collection $Z$ of vertex sets we write $V_{Z}=\bigcup_{\boldsymbol{z}\in Z}\boldsymbol{z}$, and for each $0\le i\le\n W$ and some $B$ to be determined, define
\[
\Psi_{i}=\left\{ e\left(W_{i}\cup U\cup V_{Z}\right):\;Z\subseteq M_{i},\;\left|Z\right|=\n Z,\;\left|e\left(U,V_{Z}\right)-\alpha\n Zd_{U^{0}}\right|\le B\n D\right\} .
\]
Now the significance of the quantities $e_i$ should be more clear: we expect the values in $\Psi_i$ to be about $e(U)+e_i+\alpha \n Z d_{U^0}$, so the idea is that the separation we have established between the $e_{i_j}$ should translate to the $\Psi_{i_j}$ not interfering too much with each other.

Note that we can apply \cref{lem:sub-random} to determine $\delta=\Omega\left(1\right)$
and $B=O\left(1\right)$ such that for each $0\le i\le\n W$, $\left|\Psi_{i}\right|=\Omega\left(\n Z\sqrt{\n D}\right)=\Omega\left(\n D\right)$
with probability at least $1/4$. Indeed, the first condition of \cref{lem:sub-random} follows from (i) in \cref{claim:lonely} and a sufficiently small choice of $c'$, the second condition follows from property 4 of \cref{lem:step-2}, and the third condition follows from (ii) in \cref{claim:lonely} and property 3 of \cref{lem:step-2}. We will next show that there is an outcome of $U$ for which many $\Psi_{i_j}$ are large, and in addition the cumulative deviations introduced by the randomness of $U$ do not too severely affect the separation we established so far. To this end, for each $0<i\le\n W$ define
\[
g_{i}=\left(d_{U}\left(
\vphantom{w_{\n W-i+1}^{-}}
w_{i}^{+}\right)-d_{U}\left(w_{\n W-i+1}^{-}\right)\right)-\left(\alpha d_{U^{0}}\left(
\vphantom{w_{\n W-i+1}^{-}}
w_{i}^{+}\right)-\alpha d_{U^{0}}\left(w_{\n W-i+1}^{-}\right)\right).
\]
Basically, $|g_i|$ measures the deviation of the separation $e(U,W_i)-e(U,W_{i-1})$ from its expected value $\alpha e(U^0,W_i)-\alpha e(U^0,W_{i-1})$. We will control the cumulative deviation $\sum_{i=1}^{\n W}\left|g_{i}\right|$; the absolute deviations $\left|e(U,W_i)-\alpha e(U^0,W_i)\right|$ are unfortunately too large to control directly.
\begin{claim}
\label{claim:gaps-concentration}The following hold together with
positive probability.
\begin{enumerate}
\item [(i)]There is a subset $\mathcal{J}$ of $\left(0.1\right)t$ indices
$j$ for which $\left|\Psi_{i_{j}}\right|=\Omega\left(\n D\right)$ (that is, a positive proportion of $\Psi_{i_{j}}$ are large);
\item [(ii)]$\sum_{i=1}^{\n W}\left|g_{i}\right|\le O\left(n\sqrt{\n D}\right)$.
\end{enumerate}
\end{claim}

\begin{proof}
First we show that (i) holds with probability at least $1/6$. As
discussed above, for each $1\le j\le t$, by \cref{lem:sub-random}
we have $\left|\Psi_{i_{j}}\right|=\Omega\left(\n D\right)$ with
probability at least $1/4$. Let $\overline{\mathcal{J}}$ be the
set of $j$ for which this fails, so $\E\left|\overline{\mathcal{J}}\right|\le3t/4$
and by Markov's inequality, $\left|\overline{\mathcal{J}}\right|\le\left(0.9\right)t$
with probability at least $1/6$.

Next we show that (ii) holds with probability at least $0.9$, meaning
that we can use the union bound to show that (i) and (ii) hold simultaneously
with positive probability. For this, note that for each $0<i\le\n W$,
$g_{i}$ is of $\left(\left|U^{0}\right|,\n D/\left|U^{0}\right|,O\left(1\right)\right)$-hypergeometric
type and has mean zero (because $\E d_U(w)=\alpha d_{U^0}(w)$ for any $w\in W$), so by \cref{lem:concentration} we have
\[
\Pr\left(\left|g_{i}\right|\ge r\right)\le e^{-\Omega\left(r^{2}/\n D\right)}.
\]
For a sufficiently large constant $Q$ we have
\begin{align*}
\E\left|g_{i}\right| & =\sum_{r=1}^{\infty}\Pr\left(\left|g_{i}\right|\ge r\right)\le Q\sqrt{\n D}+\sum_{r=Q\sqrt{\n D}}^{\infty}e^{-\Omega\left(r^{2}/\n D\right)}\le2Q\sqrt{\n D},\\
\E\sum_{i=1}^{\n W}\left|g_{i}\right| & \le2Q\n W\sqrt{\n D},
\end{align*}
and by Markov's inequality $\sum_{i=1}^{\n W}\left|g_{i}\right|\le20Q\n W\sqrt{\n D}=O\left(n\sqrt{\n D}\right)$
with probability at least $0.9$.
\end{proof}
Fix an outcome of $U$ such that the above properties hold.

We now take a moment to  summarise the situation so far. We have  $c'n\le \n D\le 2c'n$ and $\n Z=\Theta(\sqrt{\n D})$ for some small constant $c'$ (and the constants in all asymptotic notation are independent of $c'$). With $e_i=e\left(W_{i}\right)+\alpha e\left(U^{0},W_{i}\right)+\n Zd_{i}$, we have a subsequence of
indices $0=i_{1},\dots,i_{t}=\n W$, for $t=\Omega\left(\sqrt{n}\right)$,
such that $e_{i_{j}}-e_{i_{j-1}}\ge n$ for each $1<j\le t$. We also have matchings $M_i$ such that the degrees $d_{W_{i}}\left(\boldsymbol{v}\right)$, for $0\le i\le\n W$ and $\boldsymbol{v}\in M_{i}$, are very tightly controlled (to be precise, \cref{claim:lonely} (ii) says that $\left|d_{W_{i}}\left(\boldsymbol{v}\right)-d_{i}\right|=o\left(\sqrt{n}\right)$). Moreover, \cref{claim:gaps-concentration} shows that many $\left|\Psi_{i_{j}}\right|$ are large (specifically, $\left|\Psi_{i_{j}}\right|=\Omega\left(\n D\right)$ for $\Omega(\sqrt n)$ different $j$), and the cumulative deviation $\sum_{i=1}^{\n W}\left|g_{i}\right|\le O\left(n\sqrt{\n D}\right)$ caused by dropping to a random subset $U=U^0\setminus D$ is not too severe. We next show that many of the $\Psi_{i_{j}}$ are disjoint, which essentially completes the proof of \cref{conj:EFS}.
\begin{claim}
For sufficiently small $c'$, there is a subset $\mathcal{J}'\subseteq\mathcal{J}$
of $\Omega\left(n^{1/2}\right)$ indices $j$ among which each $\Psi_{i_{j}}$
is disjoint.
\end{claim}

\begin{proof}
For $1\le j<t$, let $\Sigma_{j}=n\left(j-1\right)-\sum_{i=1}^{i_{j}}\left|g_{i}\right|$,
so that $\Sigma_{1}=0$ and $\Sigma_{t}\ge\left(1-O\left(\sqrt{c'}\right)\right)tn$,
by  (ii) in \cref{claim:gaps-concentration}. The significance of these quantities is that we have established the separation $e_{i_j}-e_{i_{j-1}}\ge n$, but this may be offset by the buildup of deviations $|g_i|$. That is, each increment $\Sigma_{j}-\Sigma_{j-1}=n-\sum_{i=i_{j-1}+1}^{i_{j}}\left|g_{i}\right|$ is a lower bound on the separation between $e\left(W_{i_{j-1}}\right)+e\left(U,W_{i_{j-1}}\right)+ \n Zd_{i_{j-1}}$ and $e\left(W_{i_j}\right)+e\left(U,W_{i_j}\right)+\n Zd_{i_j}$, which approximates the separation between the values in $\Psi_{i_{j-1}}$ and the values in $\Psi_{i_j}$.

By \cref{lem:well-separated} (with $\tau=t=\Omega(n)$, $\lambda=\Sigma_t=\left(1-O\left(c'\right)\right)tn$, $\sigma=(0.01)t$, $\rho=n$ and $\kappa=0$), for small
$c'$ we can find an increasing sequence $j_{1}^{0},\dots,j_{s^{0}}^{0}$,
for $s^{0}\ge\left(\left(1-O\left(c'\right)\right)/1.01\right)t\ge\left(0.95\right)t$,
such that $\Sigma_{j_{q}^{0}}-\Sigma_{j_{q-1}^{0}}\ge\left(0.01\right)n$
for each $1<q\le s^{0}$. By (i) in \cref{claim:gaps-concentration},
deleting the indices not in $\mathcal{J}$ gives an increasing sequence
$j_{1},\dots,j_{s}$, with $s\ge\left(0.05\right)t$, also satisfying
$\Sigma_{j_{q}}-\Sigma_{j_{q-1}}\ge\left(0.01\right)n$ for each $1<q\le s$.

To avoid too many layered subscripts, for $1\le q\le s$ define $W_{q}'=W_{i_{j_{q}}}$, $d_{q}'=d_{i_{j_{q}}}$,
$M_{q}'=M_{i_{j_{q}}}$, $e_q'=e_{i_{j_q}}$, $i_{q}'=i_{j_{q}}$. Also, for $1<q\le s$
define $\Gamma_{q}=\sum_{i=i_{q-1}'+1}^{i_{q}'}\left|g_{i}\right|$.

Our goal is now to show that quantities of the form $e\left(W_{q}'\cup U\cup V_{Z}\right)$ arising from the definition of $|\Psi_{i_{j_q}}|$ are well-separated for different $q$. This will basically follow from the fact that $\Sigma_{j_{q}}-\Sigma_{j_{q-1}}=\Omega\left(n\right)$, our control over the $d_{W_{i}}\left(\boldsymbol{v}\right)$ for $\boldsymbol{v}\in M_{i}$, and the definition of the $\Psi_{i'_q}$.

First, for each $1<q\le s$ observe that
\begin{align*}
e\left(W_{q}',U\right)-e\left(W_{q-1}',U\right) & =\sum_{i=i_{q-1}'+1}^{i_{q}'}\left(d_{U}\left(
\vphantom{w_{\n W-i+1}^{-}}
w_{i}^{+}\right)-d_{U}\left(w_{\n W-i+1}^{-}\right)\right)\\
 & \ge\alpha\sum_{i=i_{q-1}'+1}^{i_{q}'}\left(d_{U^{0}}\left(
\vphantom{w_{\n W-i+1}^{-}}
w_{i}^{+}\right)-d_{U^{0}}\left(w_{\n W-i+1}^{-}\right)\right)-\Gamma_{q}
\\
 & =\alpha e\left(W_{q}',U^{0}\right)-\alpha e\left(W_{q-1}',U^{0}\right)-\Gamma_{q}.
\end{align*}
Next, for $Z\subseteq M_{q}'$ and $Z'\subseteq M_{q-1}'$ satisfying $\left|Z\right|=\left|Z'\right|=\n Z$
and $$\left|e\left(U,V_{Z}\right)-\alpha\n Zd_{U^{0}}\right|,\,\left|e\left(U,V_{Z'}\right)-\alpha\n Zd_{U^{0}}\right|\le B\n D=O\left(\n D\right),$$ we also have
\begin{align*}
 & e\left(W_{q}'\cup U,V_{Z}\right)-e\left(W_{q-1}'\cup U,V_{Z'}\right)+e\left(V_{Z}\right)-e\left(V_{Z'}\right)\\
 & \qquad=\sum_{\boldsymbol{v}\in Z}d_{W_{q}'}\left(\boldsymbol{v}\right)-\sum_{\boldsymbol{v}\in Z'}d_{W_{q-1}'}\left(\boldsymbol{v}\right)+O\left(\n D\right)\\
 & \qquad=\sum_{\boldsymbol{v}\in Z}\left(d_{q}'+o\left(\sqrt{\n D}\right)\right)-\sum_{\boldsymbol{v}\in Z'}\left(d_{q-1}'+o\left(\sqrt{\n D}\right)\right)+O\left(\n D\right)\\
 & \qquad=\n Zd_{q}'-\n Zd_{q-1}'+O\left(\n D\right).
\end{align*}
Recall that $\Sigma_{j_{q}}-\Sigma_{j_{q-1}}=\Omega\left(n\right)$ and $\n D\le 2c'n$. For small $c'$, it follows that
\begin{align*}
 & e\left(W_{q}'\cup U\cup V_{Z}\right)-e\left(W_{q-1}'\cup U\cup V_{Z'}\right)\\
 & \qquad=e\left(W_{q}'\right)-e\left(W_{q-1}'\right)+e\left(W_{q}',U\right)-e\left(W_{q-1}',U\right)\\
 & \qquad\quad\qquad+e\left(W_{q}'\cup U,V_{Z}\right)-e\left(W_{q-1}'\cup U,V_{Z'}\right)+e\left(V_{Z}\right)-e\left(V_{Z'}\right)\\
 & \qquad\ge\left(e\left(W_{q}'\right)+\alpha e\left(W_{q}',U^{0}\right)+\n Zd_{q}'\right)-\left(e\left(W_{q-1}'\right)+\alpha e\left(W_{q-1}',U^{0}\right)+\n Zd_{q-1}'\right)-\Gamma_{q}-O\left(\n D\right)\\
& \qquad= e_q'-e_{q-1}'-\Gamma_{q}-O\left(\n D\right)\\
 & \qquad\ge\left(j_{q}-j_{q-1}\right)n-\Gamma_{q}-O\left(\n D\right)
=\Sigma_{j_{q}}-\Sigma_{j_{q-1}}-O\left(\n D\right)=\Omega\left(n\right)>0.
\end{align*}
We conclude that the minimum value in $\Psi_{i'_{q}}$ is greater
than the maximum value in $\Psi_{i'_{q-1}}$. Since this is true
for all $1<q\le s$, it follows that each $\Psi_{i_{j_{q}}}$ is disjoint,
so we may take $\mathcal{J}'=\left\{ j_{1},\dots,j_{s}\right\} $.
\end{proof}
Finally, let $f=c'n$ and $h=\n W+k \n Z=cn+\delta\sqrt{c'n}$. For $1\le i\le\n W$
observe that if $Z\subseteq M_{i}$ satisfies $\left|Z\right|=\n Z$
then $W_{i}\cup U\cup V_{Z}$ has exactly $\left|U^{0}\right|-\n D+h$
vertices, and this number is equal to $(\ell-f)+h$ or $2(\ell-f)+h$. Therefore, for
$\ell'=(\ell-f)+h$ or $\ell'=2(\ell-f)+h$, we have
\[
\left|\Psi\left(\ell',G\right)\right|\ge\sum_{j\in\mathcal{J}'}\left|\Psi_{i_{j}}\right|=\Omega\left(c'n^{3/2}\right).\qedhere
\]
\end{proof}

\subsection{\label{sec:step-2}Proof of \texorpdfstring{\cref{lem:step-2}}{Lemma~\ref{lem:step-2}}}

As outlined in \cref{sec:discussion}, we will first construct $W^{-}$,
$W^{+}$ and $U^{0}$ satisfying properties 1 and 2, and we will then use
richness (\cref{lem:rich}) and the sunflower lemma (\cref{lem:sunflower})
to construct $M$ satisfying properties 3 and 4. We remark that it
would be possible to use an existing discrepancy theorem (for example,
a theorem in~\cite{EGPS88}, as mentioned in \cref{sec:discussion})
to construct sets $W^{-}$, $W^{+}$ and $U^{0}$ satisfying property
2, using only the fact that $G$ has density bounded away from 0 and
1. However, since we are already using \cref{lem:rich} for property
4, it is convenient to instead use richness and anticoncentration.

So, consider $\varepsilon=\varepsilon\left(C\right)$ from \cref{lem:rich},
note that we can assume $\varepsilon<1/8$, and let $\delta=\varepsilon^{K}$
for some large absolute constant $K$ which we will determine
later. Let $G\left[V'\right]$ be a $\left(\delta,\varepsilon\right)$-rich
induced subgraph of $G$, with $n':=\left|V'\right|\ge15cn$ vertices, which
exists for small $c>0$ by \cref{lem:rich}. We will only work inside
$V'$, so all degrees and neighbourhoods should be interpreted as
being restricted to $V'$.

First, let $U^{1}$ be a uniformly random subset of $V'$ with size
$2\ell\le4cn$. Let $H\subseteq\binom{V'}{2}$ be the auxiliary graph with an
edge $\left\{ x,y\right\} \in\binom{V'}{2}$ whenever $d_{U^{1}}\left(x\right)=d_{U^{1}}\left(y\right)$.
We show that with positive probability, the diversity of neighbourhoods
in $G\left[V'\right]$ is maintained for neighbourhoods in $U^{1}$,
and simultaneously $H$ is quite sparse, which implies that there
is a lot of variation between degrees into $U^{1}$ (this will be
the starting point from which we obtain our discrepancy for property
2).
\begin{claim}
\label{claim:diversity-anticoncentration}The following hold together
with positive probability.
\begin{enumerate}
\item [(i)]For each $k\le K$ and $\boldsymbol{x},\boldsymbol{y}\in\binom{V}{k}$
with $\left|N\left(\boldsymbol{x}\right)\triangle N\left(\boldsymbol{y}\right)\right|\ge\varepsilon^{K}n'$,
we have $\left|N_{U^{1}}\left(\boldsymbol{x}\right)\triangle N_{U^{1}}\left(\boldsymbol{y}\right)\right|\ge\varepsilon^{K}\ell$;
\item [(ii)]there is a set $W$ of at least $7cn$ vertices such that $d_{H}\left(x\right)=O\left(\sqrt{n}\right)$
for each $x\in W$.
\end{enumerate}
\end{claim}

\begin{proof}
We will show that (i) and (ii) each hold with probability greater
than $1/2$. The proofs will be quite routine, using the concentration and anticoncentration theorems in \cref{subsec:prob-tools}.

For (i), observe that for each $\boldsymbol{x},\boldsymbol{y}\in\binom{V}{k}$,
$\left|N_{U^{1}}\left(\boldsymbol{x}\right)\triangle N_{U^{1}}\left(\boldsymbol{y}\right)\right|=\left|\left(N\left(\boldsymbol{x}\right)\triangle N\left(\boldsymbol{y}\right)\right)\cap U^{1}\right|$
is of $\left(n',2\ell/n',1\right)$-hypergeometric type, and apply
\cref{lem:concentration} and the union bound. (Recall from \cref{subsec:definitions}
the nonstandard multiset definition of $A\triangle B$).

For (ii), note that each $d_{U^{1}}\left(x\right)-d_{U^{1}}\left(y\right)$
is of $\left(n',2\ell/n',1,\left|N\left(x\right)\triangle N\left(y\right)\right|\right)^{*}$-hypergeometric
type, so if $\left|N\left(x\right)\triangle N\left(y\right)\right|=\Omega\left(n\right)$
then by \cref{lem:anticoncentration}, $\Pr\left(d_{U^{1}}\left(x\right)=d_{U^{1}}\left(y\right)\right)=O\left(1/\sqrt{n}\right)$.
By \cref{lem:few-bad} (taking $k=1$), there are at most $n^{1/5}$
vertices $x\in V'$ with $N\left(x\right)<\varepsilon n'$, and by
\cref{lem:diversity-tuples}, for every other vertex $x\in V'$ there
are at most $n^{1/5}$
vertices $y\ne x$ with $\left|N\left(x\right)\triangle N\left(y\right)\right|<\varepsilon^{2}n'$.
For each $x\in V'$ of the latter type, we have $\E d_{H}\left(x\right)=O\left(n^{1/5}+\sqrt{n}\right)=O\left(\sqrt{n}\right)$,
so by Markov's inequality, $d_{H}\left(x\right)=O\left(\sqrt{n}\right)$
(for a sufficiently large constant implied by the big-oh notation)
with probability at least $7/8$. Let $W$ be the set of all $x\in V'$
for which this holds, so that $\E\left|V'\setminus W\right|\le n'/8+n^{1/5}<n'/4$.
Therefore, $\left|W\right|\ge n'/2\ge7cn$ with probability greater
than $1/2$.
\end{proof}
Fix an outcome of $U^{1}$ satisfying both the properties in the above
claim, and note that $\left|W\setminus U^{1}\right|\ge3cn$. Order
the vertices $x\in W\setminus U^{1}$ by their values of $d_{U^{1}}\left(x\right)$
(breaking ties arbitrarily), let $W^{1}$ contain the first $cn$
vertices in this ordering and let $W^{2}$ contain the last $cn$.
By (ii) in \cref{claim:diversity-anticoncentration}, for the (at least
$cn$) vertices $x$ between $W^{1}$ and $W^{2}$ in this ordering, we have $d_{H}\left(x\right)=O\left(\sqrt{n}\right)$,
so there are at least $\Omega\left(\sqrt{n}\right)$ values of $d_{U^{1}}\left(x\right)$, and
\[
\min_{x\in W^{2}}d_{U^{1}}\left(x\right)-\max_{x\in W^{1}}d_{U^{1}}\left(x\right)=\Omega\left(\sqrt{n}\right).
\]
Recalling that $\alpha\ge 1/2$, this implies that
\[
\alpha e\left(W^{2},U^{1}\right)-\alpha e\left(W^{1},U^{1}\right)=\Omega\left(n^{3/2}\right).
\]
Now, if 
\[
\left(e\left(W^{2}\right)+\alpha e\left(W^{2},U^{1}\right)\right)-\left(e\left(W^{1}\right)+\alpha e\left(W^{1},U^{1}\right)\right)\ge\left(\alpha e\left(W^{2},U^{1}\right)-\alpha e\left(W^{1},U^{1}\right)\right)/4
\]
then let $W^{-}=W^{1}$ and $W^{+}=W^{2}$ and $U^{0}=U^{1}$; property
2 is satisfied. Otherwise, there must be a large discrepancy between
$e\left(W^{1}\right)$ and $e\left(W^{2}\right)$. To be specific,
we must have
\begin{equation}
\left(e\left(W^{1}\right)+\alpha e\left(W^{1},U^{1}\right)/2\right)-\left(e\left(W^{2}\right)+\alpha e\left(W^{2},U^{1}\right)/2\right)\ge\left(\alpha e\left(W^{2},U^{1}\right)-\alpha e\left(W^{1},U^{1}\right)\right)/4.\label{eq:discrepancy}
\end{equation}
In this case, let $U^{0}$ be a random subset of $\ell=\left|U^{0}\right|/2$
elements of $U^{1}$, let $W^{-}=W^{2}$ and let $W^{+}=W^{1}$. Then
\begin{align*}
&\left(e\left(W^{+}\right)+\alpha e\left(W^{+},U^{0}\right)\right)-\left(e\left(W^{-}\right)+\alpha e\left(W^{-},U^{0}\right)\right)\\
&\qquad= \left(e\left(W^{+}\right)-e\left(W^{-}\right)\right)+\alpha \sum_{u\in U^0} \left(d_{W^+}(u)-d_{W^+}(u)\right)
\end{align*}
is of $\left(1/2\right)$-hypergeometric type and has mean $\Omega\left(n^{3/2}\right)$,
given by \cref{eq:discrepancy}. So, by \cref{lem:symmetric}, this
random value is $\Omega\left(n^{3/2}\right)$ with probability at
least $1/2$. Also, for each $k\le K$ and $\boldsymbol{x},\boldsymbol{y}\in\binom{V}{k}$ with
$\left|N\left(\boldsymbol{x}\right)\triangle N\left(\boldsymbol{y}\right)\right|\ge\varepsilon^{K}n'$, the random variable
$\left|N_{U^{0}}\left(\boldsymbol{x}\right)\triangle N_{U^{0}}\left(\boldsymbol{y}\right)\right|$
is of $\left(\Omega(n),1/2,1\right)$-hypergeometric type with mean $\Omega(n)$, so by
\cref{lem:concentration} and the union bound, with probability $1-o(1)$
we have $\left|N_{U^{0}}\left(\boldsymbol{x}\right)\triangle N_{U^{0}}\left(\boldsymbol{y}\right)\right|=\Omega\left(n\right)$ for all such $k,\boldsymbol x,\boldsymbol y$. So, we can fix an outcome of $U^{0}$ satisfying both of these properties.

In either of the above two cases, property 2 is satisfied and $\left|N_{U^{0}}\left(\boldsymbol{x}\right)\triangle N_{U^{0}}\left(\boldsymbol{y}\right)\right|=\Omega\left(n\right)$
for each $\boldsymbol{x},\boldsymbol{y}\in\binom{V}{k}$ with $\left|N\left(\boldsymbol{x}\right)\triangle N\left(\boldsymbol{y}\right)\right|\ge\varepsilon^{K}n'$. We also have $|U^0|=\ell$ or $|U^0|=2\ell$, satisfying property 1.
Now, fix some $\Omega\left(n\right)$-vertex subset $A^{0}$ disjoint
from $U^{1}$ and $W$, and let $M^{0}\subseteq\binom{A^{0}}{K}$
contain every $\boldsymbol{v}\in\binom{A^{0}}{K}$ with $\left|\bigcap_{v\in\boldsymbol{v}}N\left(v\right)\right|\ge\varepsilon^{K}n'$.
By \cref{lem:few-bad}, we have $\left|M^{0}\right|=\Omega\left(n^{K}\right)$.

Observe that there are only $\left(kn+1\right)^{3}$ possible values
of the tuples $\left(d_{W^{+}}\left(\boldsymbol{x}\right),d_{W^{+}}\left(\boldsymbol{x}\right),d_{U^{0}}\left(\boldsymbol{x}\right)\right)$,
so by the pigeonhole principle there are $d_{W^{-}}',d_{W^{+}}',d_{U^{0}}'\in\NN$,
and a collection $M^{1}\subseteq M^{0}$ of size $\Omega\left(n^{K-3}\right)$,
such that for each $\boldsymbol{x}\in M^{1}$ we have $d_{W^{-}}\left(\boldsymbol{x}\right)=d_{W^{-}}'$,
$d_{W^{+}}\left(\boldsymbol{x}\right)=d_{W^{+}}'$ and $d_{U^{0}}\left(\boldsymbol{x}\right)=d_{U^{0}}'$.
For sufficiently large $K$, by \cref{lem:sunflower}, $M^{1}$ has
a sunflower with $\Omega\left(n^{\left(K-3\right)/K}\right)=\Omega\left(n^{3/4+1/5}\right)$
petals; take $M^{2}$ as this set of petals, and let $k$ be the common
size of these petals. Let $\boldsymbol{v}$ be the kernel of the sunflower,
and let $d_{W^{-}}=d_{W^{-}}'-d_{W^{-}}\left(\boldsymbol{v}\right)$,
$d_{W^{+}}=d_{W^{+}}'-d_{W^{+}}\left(\boldsymbol{v}\right)$ and $d_{U^{0}}=d_{U^{0}}'-d_{U^{0}}\left(\boldsymbol{v}\right)$,
so for $\boldsymbol{x}\in M^{2}$ we have $d_{W^{-}}\left(\boldsymbol{x}\right)=d_{W^{-}}$,
$d_{W^{+}}\left(\boldsymbol{x}\right)=d_{W^{+}}$ and $d_{U^{0}}\left(\boldsymbol{x}\right)=d_{U^{0}}$.

Finally, consider the auxiliary graph $F\subseteq\binom{M^{2}}{2}$ which has an
edge $\left\{ \boldsymbol{x},\boldsymbol{y}\right\} \in\binom{M^{2}}{2}$
whenever $\left|N\left(\boldsymbol{x}\right)\triangle N\left(\boldsymbol{y}\right)\right|<\varepsilon^{K}n'$.
By \cref{lem:diversity-tuples}, the degrees in $F$ are at most $n^{1/5}$
so by \cref{prop:turan} (Tur\'an's theorem) there is $M\subseteq M^{2}$ with $\left|M\right|=\Omega\left(n^{3/4}\right)$
such that $\left|N\left(\boldsymbol{x}\right)\triangle N\left(\boldsymbol{y}\right)\right|\ge\varepsilon^{K}n'$,
and therefore $\left|N_{U^{0}}\left(\boldsymbol{x}\right)\triangle N_{U^{0}}\left(\boldsymbol{y}\right)\right|=\Omega\left(n\right)$,
for all pairs $\left\{ \boldsymbol{x},\boldsymbol{y}\right\} \in\binom{M}{2}$.

\subsection{Proof of \texorpdfstring{\cref{lem:sub-random}}{Lemma~\ref{lem:sub-random}}}\label{sec:sub-random}

As in the deduction of \cref{conj:EFS} in \cref{sec:EFS-proof}, for
a collection $Z$ of vertex sets let $V_{Z}=\bigcup_{\boldsymbol{z}\in Z}\boldsymbol{z}$.

Our proof of \cref{lem:sub-random} will be quite similar to the proof
of the main theorem in~\cite{KS}. Roughly speaking, we will first
expose a random superset $D^{1}$ of $D$ (we may view this as ``partially
exposing'' the random subset $D$). Using this randomness for anticoncentration,
we will construct sub-matchings $S^{-},S^{+}\subseteq M$ of size
$\Omega\left(\sqrt{\n D}\right)$, such that all the degrees from
elements of $S^{+}$ into $D^{1}$ are higher by $\sqrt{\n D}$
than the degrees from $S^{-}$ into $D^{1}$. Starting with any $S_{0}\subseteq S^{+}$
of some size $\n Z-1$, we can therefore obtain $\n Z$ subsets $S_{0},\dots,S_{\n Z-1}$
such that the values $e\left(W\cup U^{0}\cup V_{S_{i}}\right)$ are
separated by a distance of $\Omega\left(\sqrt{\n D}\right)$, simply
by switching elements of $M$ one-by-one from $S^{+}$ into $S^{-}$. Then,
we fully expose the random set $D$ (therefore exposing $U=U^0\setminus D$), and show that the values $e\left(W\cup U\cup V_{S_{i}}\right)$
remain fairly well-separated. We use this further randomness, and
anticoncentration, to show that for most $i$, there is a set $X_{i}$
of $\sqrt{\n D}$ elements of $M$ which have different degrees into
$W\cup U\cup V_{S_{i}}$, still concentrated in a known interval of
length $O\left(\sqrt{\n D}\right)$. This will prove that there are
$\Omega\left(\n Z\sqrt{\n D}\right)$ values $e\left(W\cup U\cup V_{S_{i}}\cup\boldsymbol{z}\right)$,
for $\boldsymbol{x}\in X_{i}$. (So, our sets $Z$ in the lemma statement are of the form $S_{i}\cup\{\boldsymbol x\}$, for $\boldsymbol x\in X_i$).
The additional requirement that there are about the expected number of edges between $U$ and $Z$ will follow from our proof basically for free.

We now proceed with this plan to prove \cref{lem:sub-random}. Arbitrarily
split $M$ into two subsets $S^{0}$ and $X^{0}$ each of size $\Omega\left(\sqrt{\n D}\right)$.
Let $D^{1}$ be a uniformly random subset of $U^{0}$ of size $2\n D$, so that we may realise the desired distribution of $D$ as a uniformly random subset of $D^{1}$ of size $\n D$. We
will first observe some regularity and discrepancy properties that
hold with probability at least $3/4$ with respect to the random choice
of $D^{1}$. Let $H\subseteq\binom{S^{0}}{2}$ be the auxiliary random graph
(depending on $D^{1}$) with an edge $\left\{ \boldsymbol{x},\boldsymbol{y}\right\} \in\binom{S^{0}}{2}$ if $d_{D^1}(\boldsymbol x)=d_{D^1}(\boldsymbol y)$. Also, let $d_{D}=(1-\alpha) d_{U^{0}}$, recalling from the statement of \cref{lem:sub-random} that $1-\alpha=\n D/\left|U^{0}\right|$.
\begin{claim}
\label{claim:X-S-split}The following hold together with probability
at least $3/4$.
\begin{enumerate}
\item [(i)]$\left|N_{D^{1}}\left(\boldsymbol{x}\right)\triangle N_{D^{1}}\left(\boldsymbol{y}\right)\right|=\Omega\left(\n D\right)$
for each $\left\{ \boldsymbol{x},\boldsymbol{y}\right\} \in\binom{X^0}{2}$;
\item [(ii)]there are $X\subseteq X^{0}$ and $S^{1}\subseteq S^{0}$,
each with size $\Omega\left(\sqrt{\n D}\right)$, such that $d_{D^{1}}\left(\boldsymbol{x}\right)=2d_{D}+O\left(\sqrt{\n D}\right)$
for each $\boldsymbol{x}\in X\cup S^1$;
\item [(iii)]$H$ has $O\left(\sqrt{\n D}\right)$ edges.
\end{enumerate}
\end{claim}

\begin{proof}
We will prove that each of (i)-(iii) individually hold with high probability, then apply the union bound. The proofs will be rather routine, using the concentration and anticoncentration theorems in \cref{subsec:prob-tools} in a similar way to the proof of \cref{claim:diversity-anticoncentration}.

For (i), observe that for each $\left\{ \boldsymbol{x},\boldsymbol{y}\right\} \in\binom{X^0}{2}$,
the random variable
$$\left|N_{D^{1}}\left(\boldsymbol{x}\right)\triangle N_{D^{1}}\left(\boldsymbol{y}\right)\right|=\left|N_{U^{0}}\left(\boldsymbol{x}\right)\triangle N_{U^{0}}\left(\boldsymbol{y}\right)\cap D^{1}\right|$$
is of $\left(\left|U^{0}\right|,2(1-\alpha),1\right)$-hypergeometric type
with mean $\Omega\left(\n D\right)$, so by the second assumption
of this lemma, \cref{lem:concentration} and the union bound, (i) holds
with probability $1-\left|X^0\right|^{2}e^{-\Omega\left(\n D\right)}=1-o\left(1\right)$.

We next show that (ii) holds with probability at least $7/8$. For
each $\boldsymbol{x}\in X^{0}$, the random variable $d_{D^{1}}\left(\boldsymbol{x}\right)$
is of $\left(\left|U^{0}\right|,2(1-\alpha),k\right)$-hypergeometric type,
so by \cref{lem:concentration} (with $t$ a large multiple of $\sqrt{\n D}$), with probability at least $31/32$
we have $d_{D^{1}}\left(\boldsymbol{x}\right)=\E d_{D^{1}}\left(\boldsymbol{x}\right)+O\left(\sqrt{\n D}\right)=2d_{D}+O\left(\sqrt{\n D}\right)$.
Therefore the expected number of $\boldsymbol{x}\in X^{0}$ failing
to satisfy this bound is at most $\left|X^{0}\right|/32$, and the
probability more than $\left|X^{0}\right|/2$ fail to satisfy this
bound is at most $1/16$. If this does not occur, we can find an appropriate
subset $X\subseteq X^{0}$ of size $\left|X^{0}\right|/2$. A very
similar argument shows that an appropriate subset $S^{1}\subseteq S^{0}$
with size $\left|S^{0}\right|/2$ exists with probability at least
$15/16$, and by the union bound we can simultaneously find suitable
$X,S^{1}$ with probability at least $7/8$.

Finally, we show that (iii) holds with probability at least $15/16$.
This will suffice to apply the union bound over parts (i)-(iii). Note that
the random variable
$d_{D^1}(\boldsymbol x)-d_{D^1}(\boldsymbol y)
$
is of $\left(|U^0|,2(1-\alpha),O\left(1\right),\left|N_{U^{0}}\left(\boldsymbol{x}\right)\triangle N_{U^{0}}\left(\boldsymbol{y}\right)\right|\right)^{*}$-hypergeometric
type. Recalling the second assumption
of this lemma that $\left|N_{U^{0}}\left(\boldsymbol{x}\right)\triangle N_{U^{0}}\left(\boldsymbol{y}\right)\right|=\Omega\left(\left|U^{0}\right|\right)$, we may apply \cref{lem:anticoncentration} to see that for any $\left\{ \boldsymbol{x},\boldsymbol{y}\right\} \in\binom{S^{0}}{2}$,
the probability $\left\{ \boldsymbol{x},\boldsymbol{y}\right\} $
is an edge in $H$ is $O\left(1/\sqrt{(1-\alpha) |U^0|}\right)=O\left(1/\sqrt{\n D}\right)$, and the expected
number of edges is $O\left(\sqrt{\n D}\right)$. The desired result
then follows from Markov's inequality.
\end{proof}
Condition on an outcome of $D^{1}$ satisfying all the above properties
(we will treat $D^{1}$ as fixed for the remainder of the proof). By \cref{prop:turan}, the graph $H$ (which has $|S^0|=\Omega(\sqrt{\n D})$ vertices) has an independent set $S^{2}$ of size $\Omega\left(\sqrt{\n D}\right)$, meaning that the values of $d_{D^1}\left(\boldsymbol{x}\right)$,
for $\boldsymbol{x}\in S^{2}$, are all different. Now, let $\n S=\n Z-1$, and note that for small $\delta$ we
have $\n S<\delta\sqrt{\n D}\le\left|S^{2}\right|/3$. Order the vertices $\boldsymbol{x}\in S^{2}$
by their values of $d_{D^1}\left(\boldsymbol{x}\right)$, let
$S^{-}$ contain the first $\n S$ elements of this ordering and let
$S^{+}$ contain the last $\n S$ elements. By construction,
we have
\begin{equation}
\min_{\boldsymbol{x}\in S^{+}}d_{D^1}\left(\boldsymbol{x}\right)-\max_{\boldsymbol{x}\in S^{-}}d_{D^1}\left(\boldsymbol{x}\right)=\Theta\left(\sqrt{\n D}\right).\label{eq:expected-degree-difference}
\end{equation}
(Here and from now on, the constants implied by all asymptotic notation
are independent of $\delta$).

Now, fix orderings $\boldsymbol{v}_{1}^{-},\dots,\boldsymbol{v}_{\n S}^{-}$
of $S^{-}$ and $\boldsymbol{v}_{1}^{+},\dots,\boldsymbol{v}_{\n S}^{+}$
of $S^{+}$. For $0\le i\le\n S$, define
\[
S_{i}=\left\{ \vphantom{\boldsymbol{v}_{\n S-i}^{+},} \boldsymbol{v}_{1}^{-},\dots,\boldsymbol{v}_{i}^{-}\right\} \cup\left\{ \boldsymbol{v}_{1}^{+},\dots,\boldsymbol{v}_{\n S-i}^{+}\right\} ,
\]
let $U_{i}=W\cup U\cup V_{S_{i}}$, and let $e_{i}=e\left(V_{S_{i}}\right)+e\left(V_{S_{i}},W\cup U\right)=e\left(U_{i}\right)-e\left(W\cup U\right)$.
For $0<i\le\n S$ define
\begin{align}
\Delta_{i}&=e_{i}-e_{i-1}\notag\\
& = e\left(V_{S_{i}},W\cup U\right)-e\left(V_{S_{i-1}},W\cup U\right)+e\left(V_{S_{i}}\right)-e\left(V_{S_{i-1}}\right)\notag\\
 &=d_{W\cup U}\left(
\vphantom{\boldsymbol{v}_{\n S-i+1}^{+}}
\boldsymbol{v}_{i}^{-}\right)-d_{W\cup U}\left(\boldsymbol{v}_{\n S-i+1}^{+}\right)
+e\left(V_{S_{i}}\right)-e\left(V_{S_{i-1}}\right)\notag\\
 &=\left(
d_{W\cup U^{0}}\left(
\vphantom{\boldsymbol{v}_{\n S-i+1}^{+}}
\boldsymbol{v}_{i}^{-}\right)-d_{D}\left(
\vphantom{\boldsymbol{v}_{\n S-i+1}^{+}}
\boldsymbol{v}_{i}^{-}\right)\right)-\left(d_{W\cup U^{0}}\left(\boldsymbol{v}_{\n S-i+1}^{+}\right)-d_{D}\left(\boldsymbol{v}_{\n S-i+1}^{+}\right)\right)
+e\left(V_{S_{i}}\right)-e\left(V_{S_{i-1}}\right).\label{eq:delta-i}
\end{align}
Next we observe
that with probability at least $1/3$, our discrepancy properties
are to some extent maintained, while for many $i$ we can find many
vertices in $X$ with distinct degrees into $U_{i}$. Recall that $D$ is a random subset of half the elements of $|D^1|$.
\begin{claim}
\label{claim:final-different-degrees}There are $\gamma_{1},\gamma_{3}=\Omega\left(1\right)$
and $Q_{2},Q_{4}=O\left(1\right)$ such that the following hold together
with probability at least $1/3$.
\begin{enumerate}
\item [(i)]there is a set $\mathcal{I}_{1}$ of $\left(1-\gamma_{1}/\left(8Q_{2}\right)\right)\n S$
indices $i$ such that for each $i\in\mathcal{I}_{1}$, we have $e\left(D,V_{S_{i}}\right)=\n Sd_{D}+O\left(\n D\right)$;
\item [(ii)]There is a set $\mathcal{I}_{2}$ of $\left(1-\gamma_{1}/\left(8Q_{2}\right)\right)\n S$
indices $i$, each with a set $X_{i}\subseteq X$ of size $2\gamma_{3}\left|X\right|$,
such that the $d_{U_{i}}\left(\boldsymbol{x}\right)$, for $\boldsymbol{x}\in X_{i}$,
are distinct;
\item [(iii)]there is a set $X^{*}$ of size $\left(1-\gamma_{3}\right)\left|X\right|$
such that for each $\boldsymbol{x}\in X^{*}$ we have $\left|d_{D}-d_{D}\left(\boldsymbol{x}\right)\right|\le Q_{4}\sqrt{\n D}$;
\item [(iv)]$e_{\n S}-e_{0}\ge3\gamma_{1}\n S\sqrt{\n D}$;
\item [(v)]$\sum_{i:\left|\Delta_{i}\right|\ge Q_{2}\sqrt{\n D}}\left|\Delta_{i}\right|\le\gamma_{1}\n S\sqrt{\n D}.$
\end{enumerate}
\end{claim}

\begin{proof}
We will prove that each part holds with probability at least $0.99$,
except (iv), which holds with probability at least $1/2$. The values of $\gamma_{1},Q_{2},\gamma_{3},Q_{4}$
will be determined in order, and will depend on each other.

For (iv), recalling \cref{eq:delta-i} we observe
\begin{align*}
\E\,\Delta_i & = \E\left[e_{i}-e_{i-1}\right]\\
 &=\left(
d_{W\cup U^{0}}\left(
\vphantom{\boldsymbol{v}_{\n S-i+1}^{+}}
\boldsymbol{v}_{i}^{-}\right)-d_{D^1}\left(
\vphantom{\boldsymbol{v}_{\n S-i+1}^{+}}
\boldsymbol{v}_{i}^{-}\right)/2\right)-\left(d_{W\cup U^{0}}\left(\boldsymbol{v}_{\n S-i+1}^{+}\right)-d_{D^1}\left(\boldsymbol{v}_{\n S-i+1}^{+}\right)/2\right)-O\left(\n S\right).
\end{align*}

Recall from the third assumption of this lemma that $d_{W\cup U^{0}}\left(\boldsymbol{v}\right)=d_{U^{0}}+d_W+o\left(\sqrt{\n D}\right)$ for all $\boldsymbol{v}\in M$, and recall from \cref{eq:expected-degree-difference} that the degrees from $S^+$ into $D^1$ are larger by $\Theta\left(\sqrt{\n D}\right)$ than the degrees from $S^-$ into $D^1$. Also, recall that $\n S<\delta \sqrt{\n D}$. For small $\delta$ it follows that
$$\E\left[e_{i}-e_{i-1}\right] = \Theta\left(\sqrt{\n D}\right)-o\left(\sqrt{\n D}\right)-O\left(\n S\right)=\Theta\left(\sqrt{\n D}\right).$$
So, $\E\left[e_{\n S}-e_{0}\right]=\Theta\left(\n S\sqrt{\n D}\right)$.
Since $e_{\n S}-e_{0}$ is of $\left(1/2\right)$-hypergeometric type, we may apply \cref{lem:symmetric} to show that for small $\gamma_{1}$ it is at least as
large as its expectation $\Omega\left(\n S\sqrt{\n D}\right)\ge3\gamma_{1}\n S\sqrt{\n D}$,
with probability at least $1/2$.

For (v), observe that for each $0<i\le\n S$, the random variable
$\Delta_{i}$ is of $\left(2\n D,1/2,k\right)$-hypergeometric type, because it is a translation of the random variable $d_{D}\left(\boldsymbol{v}_{\n S-i+1}^{+}\right)-d_{D}\left(
\vphantom{\boldsymbol{v}_{\n S-i+1}^{+}}
\boldsymbol{v}_{i}^{-}\right).$
We have just computed that $\E\,\Delta_i=O\left(\sqrt{\n D}\right)$, so by \cref{lem:concentration} we therefore have $\Pr\left(\left|\Delta_{i}\right|\ge t\right)=\exp\left(-\Omega\left(t^{2}/\n D\right)\right)$.
Now, for any nonnegative integer random variable $\xi$, we have $\E\xi=\sum_{t=1}^{\infty}\Pr\left(\xi\ge t\right)$,
so
\begin{align*}
\E\left[\left|\Delta_{i}\right|\one_{\left|\Delta_{i}\right|\ge Q_{2}\sqrt{\n D}}\right] & =\sum_{t=1}^{\infty}\Pr\left(\left|\Delta_{i}\right|\one_{\left|\Delta_{i}\right|\ge Q_{2}\sqrt{\n D}}\ge t\right)\\
 & =Q_{2}\sqrt{\n D}\Pr\left(\left|\Delta_{i}\right|\ge Q_{2}\sqrt{\n D}\right)+\sum_{t=Q_{2}\sqrt{\n D}}^{\infty}\Pr\left(\left|\Delta_{i}\right|\ge t\right)\\
 & =Q_{2}\sqrt{\n D}e^{-\Omega\left(Q_{2}^{2}\right)}+\sum_{t=Q_{2}\sqrt{\n D}}^{\infty}\exp\left(-\Omega\left(t^{2}/\n D\right)\right)=e^{-\Omega\left(Q_{2}^{2}\right)}\sqrt{\n D}.
\end{align*}
For sufficiently large $Q_{2}$, this is at most $\left(\gamma_{1}/100\right)\sqrt{\n D}$,
so
\[
\E\sum_{i:\left|\Delta_{i}\right|\ge Q_{2}\sqrt{\n D}}\left|\Delta_{i}\right|\le\left(\gamma_{1}/100\right)\n S\sqrt{\n D}
\]
and (v) holds with probability at least $0.99$ by Markov's inequality.

For (i), recall from (ii) of
\cref{claim:X-S-split} that each $\boldsymbol x\in S^1$ has degree $2d_D+O\left(\sqrt{\n D}\right)$ into $D^1$. Therefore, for each $0\le i\le\n S$, $e\left(D,V_{S_i}\right)$
is of $\left(2\n D,1/2,O\left(\sqrt{\n D}\right)\right)$-hypergeometric
type, and has mean $\n Sd_{D}+O\left(\n S \sqrt \n D\right)=\n Sd_{D}+O(\n D)$. So, applying \cref{lem:concentration} with $t$ a large multiple of $\n D$, we have $e\left(D,V_{Z}\right)=\n Sd_{D}+O\left(\n D\right)$
with probability at least $1-\gamma_{1}/\left(800Q_{2}\right)$. The expected
number of indices $i$ for which this fails is $\left(\gamma_{1}/\left(800Q_{2}\right)\right)\n S$,
so by Markov's inequality, the probability it fails for more than
$\left(\gamma_{1}/\left(8Q_{2}\right)\right)\n S$ indices $i$ is
at most $0.99$.

Next we consider (ii). For each $i$ and each $\left\{ \boldsymbol{x},\boldsymbol{y}\right\} \in\binom{X}{2}$, let
$$d_i=\left(d_W(\boldsymbol x)+d_{U^0}(\boldsymbol x)+d_{V_{S_i}}(\boldsymbol x)\right)-\left(d_W(\boldsymbol y)+d_{U^0}(\boldsymbol y)+d_{V_{S_i}}(\boldsymbol y)\right)=o(\sqrt{\n D})+O(\n Z),$$
so $|d_i|\le \sqrt{\n D}$ for small $\delta$. Then, observe that the random variable
$$d_{U_{i}}\left(\boldsymbol{x}\right)-d_{U_{i}}\left(\boldsymbol{y}\right)-d_i=d_D(\boldsymbol y)-d_D(\boldsymbol x)$$
is of $\left(2\n D,1/2,O\left(1\right),\left|N_{D^{1}}\left(\boldsymbol{x}\right)\triangle N_{D^{1}}\left(\boldsymbol{y}\right)\right|\right)^{*}$-hypergeometric
type. So, by part (i) of \cref{claim:X-S-split} and \cref{lem:anticoncentration}, $\Pr\left(d_{U_{i}}\left(\boldsymbol{x}\right)=d_{U_{i}}\left(\boldsymbol{y}\right)\right)=O\left(1/\sqrt{\n D}\right)$.
Let $H_{i}$ be the graph of pairs $\left\{ \boldsymbol{x},\boldsymbol{y}\right\} \in\binom{X}{2}$
satisfying $d_{U_{i}}\left(\boldsymbol{x}\right)=d_{U_{i}}\left(\boldsymbol{y}\right)$,
so we have $\E e\left(H_{i}\right)=O\left(\sqrt{\n D}\right)$. By
Markov's inequality, with probability at least $1-\gamma_{1}/\left(800Q_{2}\right)$
we have $e\left(H_{i}\right)=O\left(\sqrt{\n D}\right)$, in which
case by \cref{prop:turan}, $H_{i}$ has an independent set $X_{i}$
of size $2\gamma_{3}\sqrt{n}$, for some $\gamma_{3}>0$. The expected
proportion of indices $i$ for which this fails to occur is $\gamma_{1}/\left(800Q_{2}\right)$,
and by Markov's inequality again, with probability at least $0.99$
it fails for only a $\gamma_{1}/\left(8Q_{2}\right)$ proportion.

Finally we consider (iii). For each $\boldsymbol{x}\in X$, $d_{D}\left(\boldsymbol{x}\right)$ is of ($2\n D$, $1/2$, $O(1)$)-hypergeometric type, and by (ii) in
\cref{claim:X-S-split}, it has mean $d_D+O(\sqrt{\n D})$. Therefore, by \cref{lem:concentration}, with large
enough $Q_{4}$, we have $\left|d_{D}-d_{D}\left(\boldsymbol{x}\right)\right|\le Q_{4}\sqrt{\n D}$
with probability at least $1-\gamma_{3}/100$, and by Markov's inequality
the probability this fails for more than $\gamma_{3}\left|X\right|$
vertices is at most $0.99$.
\end{proof}
Now it is a relatively simple matter to put everything together to prove \cref{lem:sub-random}. Fix $\gamma_{1},Q_{2},\gamma_{3},Q_{4}$ and $U$ such that all parts
of the above claim are satisfied. By (iii), for any $0\le i\le\n S$, any $\boldsymbol{x}\in X^{*}$, and small $\delta$, we have
\begin{equation}
\left|d_{U_{i}}\left(\boldsymbol{x}\right)-\left(\alpha d_{U}+d_{W}\right)\right|\le d_{V_{S_i}}(\boldsymbol x)+Q_{4}\sqrt{\n D}+o(\sqrt {\n D})= O(\n S)+Q_{4}\sqrt{\n D}< 2Q_{4}\sqrt{\n D}.\label{eq:separation}
\end{equation}

By \cref{lem:well-separated} (with $\lambda = 3\gamma_1 \n S \sqrt{\n D}$, $\rho=Q_2\sqrt{\n D}$, $\kappa=\gamma_1 \n S\sqrt{\n D}$ and $\sigma=\sqrt{\n D}$) and parts (iv) and (v) of the above claim,
for large enough $Q_{2}$ there is an increasing subsequence $i_{1},\dots,i_{t}$,
with $t\ge\gamma_{1}\n S/\left(2Q_{2}\right)$, such that $e_{i-1}-e_{i}\ge\sqrt{\n D}$
for each $1<i\le t$. Delete all indices not in $\mathcal{I}_{1}\cap\mathcal{I}_{2}$ (there are at most $\gamma_1 \n S/(4Q_2)$ such)
to obtain a subsubsequence $i_{1}',\dots,i_{s}'$ with $s\ge\gamma_{1}\n S/\left(4Q_{2}\right)$.
Let $\mathcal{I}$ contain every $4Q_{4}$th element of this subsubsequence,
so that $\left|\mathcal{I}\right|=\Theta\left(\n S\right)=\Theta\left(\n Z\right)$
and
$$\left|e_{i}-e_{i'}\right|=\left|e(U_i)-e(U_{i'})\right|\ge4Q_{4}\sqrt{\n D}$$
for every pair
of distinct indices $i,i'\in\mathcal{I}$. Recalling \cref{eq:separation}, this means that for different $i\in \mathcal I$, there is no overlap between the sets of values $\{e(U_i)+d_{U_i}(\boldsymbol x):\boldsymbol x\in X^*\}$. By the definition of $X_i$ in (ii) of \cref{claim:final-different-degrees}, this means that for each of the $\Theta\left(\n Z\sqrt{\n D}\right)$
choices of $i\in\mathcal{I}$ and $\boldsymbol{x}\in X_{i}\cap X^{*}$,
the values
$e\left(W\cup U\cup V_{S_{i}\cup\left\{ \boldsymbol{x}\right\} }\right)=e(U_i)+d_{U_i}(\boldsymbol x)$
are in fact distinct. It remains to show that the $e\left(U,V_{S_{i}\cup\left\{ \boldsymbol{x}\right\} }\right)$ are close to their expectations $\alpha\n Zd_{U^{0}}$. We have $e\left(U^{0},V_{S_{i}\cup\left\{ \boldsymbol{x}\right\} }\right)=\n Z d_{U^0}$, $d_{D}=(1-\alpha)d_{U^0}$ and $\n S=\n Z-1$, so by (i) and (iii) in \cref{claim:final-different-degrees}, for sufficiently large $B$,
\begin{align*}
\left|e\left(U,V_{S_{i}\cup\left\{ \boldsymbol{x}\right\} }\right)-\alpha\n Zd_{U^{0}}\right| & =\left|e\left(U^{0},V_{S_{i}\cup\left\{ \boldsymbol{x}\right\} }\right)-e\left(D,V_{S_{i}}\right)-d_{D}\left(\boldsymbol{x}\right)-\n Zd_{U^{0}}+\n Sd_{D}+d_{D}\right|\\
 & \le O\left(\n D+\sqrt{\n D}\right)\le B\n D.
\end{align*}
We have proved that the statements of \cref{claim:X-S-split,claim:final-different-degrees}
hold together with probability at least $\left(3/4\right)\left(1/3\right)=1/4$,
in which case the desired conclusion holds.

\section{Concluding remarks}\label{sec:concluding}

We have proved the Erd\H os--Faudree--S\'os
conjecture that for any fixed $C$, if $G$ is an $n$-vertex graph
with no homogeneous subgraph on $C\log n$ vertices, then $G$ contains
$\Omega\left(n^{5/2}\right)$ induced subgraphs, no pair of which
have the same numbers of vertices and edges. We feel that this area
is still a long way from maturity, and there is much more room for
further research towards understanding the structure of $C$-Ramsey
graphs. We hope that such research will inform future work on explicit
constructions of Ramsey graphs.

Regarding specific open questions, of course the Erd\H os--McKay
conjecture remains an intriguing problem. We would also like
to draw attention to the subject of subgraphs with many different
degrees: as mentioned in the introduction, answering a different conjecture of Erd\H os, Faudree and
S\'os~\cite{Erd92,Erd97}, Bukh and Sudakov~\cite{BS07} proved that
$C$-Ramsey graphs have induced subgraphs with $\Omega\left(\sqrt{n}\right)$
different degrees. However, in random graphs one can actually find
induced subgraphs with $\Omega\left(n^{2/3}\right)$ distinct degrees
(this was proved in an unpublished paper of Conlon, Morris, Samotij
and Saxton~\cite{CMSS}), and it is not clear whether such an improved
bound also holds for $C$-Ramsey graphs.

Additionally, observe that the main result of this paper can be rephrased as the fact that in an $O(1)$-Ramsey graph, for most choices of $\ell$, there are many possibilities for the number of edges in a subset of $\ell$ vertices. We believe a natural next step would be to study statistical properties of the number of edges in a \emph{random} set of $\ell$ vertices. For example, is this random variable anticoncentrated? For general graphs this question was first studied by Alon, Hefetz, Krivelevich and Tyomkyn~\cite{AHKT} (see \cite{KST,FS,MMNT} for further work). Regarding Ramsey graphs, as we recently proposed in a paper with Tuan Tran \cite{KST}, could it be true that in any $O(1)$-Ramsey graph $G$, if $A$ is a uniformly random set of $n/2$ vertices, then $\Pr(e(G[A])=x)=O(1/n)$ for all $x$? In \cite{KS} we also formulated a version of this question for random subsets where the presence of each vertex is chosen independently, which may be more tractable.

Finally, we believe an interesting further direction of research would
be to consider regimes where larger homogeneous subgraphs are forbidden
(see~\cite{AB89,AB07,AKS03,NT17} for some examples of theorems of
this type). In~\cite{KS} we proposed the conjecture that $\left|\Phi\left(G\right)\right|=\Omega\left(e\left(G\right)\right)$
for graphs $G$ which have no homogeneous subgraph on $n/4$ vertices;
we do not know a good counterpart of this conjecture for $\left|\Psi\left(G\right)\right|$,
but it seems likely that some nontrivial bound should hold.

\vspace{0.3cm}
\noindent
{\bf Acknowledgment.}\, The authors would like to thank the referee for their careful reading of the manuscript and their valuable comments. We would also like to thank Mantas Baksys and Xuanang Chen for carefully reading the paper and finding an oversight in the proof (related to the definition of richness in \cref{subsec:diversity}).

\end{document}